\documentclass[11pt]{amsart}
\usepackage{geometry}                
\usepackage{amsmath,amsthm,amssymb,latexsym,epic,bbm,comment}
\usepackage{tikz-cd}
\usepackage{graphicx,enumerate,stmaryrd, xcolor, todonotes}
\usepackage[all,2cell]{xy}
\usepackage[active]{srcltx}
\usepackage[parfill]{parskip}
\UseTwocells
\usepackage[
colorlinks=true,
linkcolor=black, 
anchorcolor=black,
citecolor=black, 
urlcolor=black, 
]{hyperref}
\usepackage[e]{esvect}
\usepackage{todonotes}
\usepackage{hyperref}
\usepackage{enumerate}

\newtheorem{theorem}{Theorem}[section]

\newtheorem{lemma}[theorem]{Lemma}
\newtheorem{proposition}[theorem]{Proposition}

\newtheorem{corollary}[theorem]{Corollary}
\theoremstyle{definition}

\newtheorem{remark}[theorem]{Remark}

\newcommand{\mone}{{\,\text{-}1}}

\newcommand{\eps}{\varepsilon}

\font\sc=rsfs10
\newcommand{\cC}{\sc\mbox{C}\hspace{1.0pt}}
\newcommand{\cH}{\sc\mbox{H}\hspace{1.0pt}}

\newcommand{\cG}{\sc\mbox{G}\hspace{1.0pt}}

\newcommand{\cR}{\sc\mbox{R}\hspace{1.0pt}}
\newcommand{\cI}{\sc\mbox{I}\hspace{1.0pt}}

\newcommand{\cB}{\sc\mbox{B}\hspace{1.0pt}}

\newcommand{\cV}{\sc\mbox{V}\hspace{1.0pt}}
\newcommand{\cX}{\sc\mbox{X}\hspace{1.0pt}}

\newcommand{\Hom}{\mathrm{Hom}}

\newcommand{\End}{\mathrm{End}}

\newcommand{\lmod}{\text{-}\mathrm{mod}}

\newcommand{\add}{\operatorname{add}}

\newcommand{\Id}{\operatorname{Id}}

\newcommand{\rF}{\mathrm{F}}

\newcommand{\J}{\mathcal{J}}
\def\H{{\mathcal{H}}}
\def\L{{\mathcal{L}}}

\newcommand{\R}{\mathcal{R}}
\def\S{\mathcal{S}}
\newcommand{\X}{\mathcal{X}}
\newcommand{\Y}{\mathcal{Y}}

\newcommand{\id}{\mathrm{id}}
\newcommand{\one}{\mathbbm{1}}

\newcommand{\bfM}{\mathbf{M}}

\newcommand{\bfC}{\mathbf{C}}

\newcommand{\ti}{\mathtt{i}}

\title{Basic Hopf algebras and symmetric bimodules}
\author{Katerina Hristova, Vanessa Miemietz}

\begin{document}
\maketitle

\begin{abstract}
Motivated by the so-called $\H$-cell reduction theorems, we investigate certain classes of bicategories which have only one $\H$-cell apart from possibly the identity. We show that $\H_0$-simple quasi fiab bicategories with unique $\H$-cell $\H_0$ are fusion categories. We further study two classes of non-semisimple quasi-fiab bicategories with a single $\H$-cell apart from the identity. The first is $\cH_A$, indexed by a  finite-dimensional radically graded basic Hopf algebra $A$, and the second is $\cG_A$, consisting of symmetric projective $A$-$A$-bimodules. We show that $\cH_A$ can be viewed as a $1$-full subbicategory of $\cG_A$  and classify simple transitive birepresentations for $\cG_A$. We point out that the number of equivalence classes of the latter is finite, while that for $\cH_A$ is generally not.
\end{abstract}

\section{Introduction}
 
Since the begining of the century, tremendous progress in representation theory has been made using ideas of categorification, see e.g. \cite{CR, EW, W}. Modelled on $2$-categories appearing in relevant examples, finitary and fiat $2$-categories as well as their finitary $2$-representations were defined in \cite{MM1, MM2}. These should be viewed as $2$-categorical analogues of finite-dimensional algebras and their finite-dimensional representations. The notion simple transitive $2$-representation, which provides an appropriate analogue for simple representations of an algebra, was introduced in \cite{MM5}, leading to the fundamental problem of classifying these for specific classes of fiat $2$-categories, as well as of developing methods to aid such a classification for an arbitrary fiat $2$-category.

The most powerful tool for classifying simple transitive $2$-representations of a general fiat $2$-category $\cC$ is given by the so-called $\H$-cell reduction results in \cite{MMMZ} and \cite{MMMTZ2}. These reduce the problem to the classification of simple transitive $2$-representations of certain subquotients $\cC_\H$ of $\cC$. These subquotients have the desirable property that they have only one object, and only a single left/right and two-sided cell (which is then an $\H$-cell) $\H_0$ of indecomposable $1$-morphisms, apart from possibly one other such cell containing only the identity $1$-morphism $\one$. Moreover, $\cC_\H$ is $\H_0$-simple. Analogous results hold for fiab bicategories, which is the approach we take in this article.

This motivates the question what shape quasi fiab bicategories with either only one $\H$-cell $\H_0$ or precisely two $\H$-cells  $\{\one\}$ and $\H_0$, take and whether these can be classified in any way, assuming they are $\H_0$-simple. The results in this paper should be seen as an initial step in this direction.

As a first result, we show that, if $\one\in \H_0$, then such an $\H_0$-simple bicategory is a fusion category. 
We then investigate two classes of quasi-fiab bicategories with precisely two left/right and two-sided cells $\{1\}$ and $\H_0$. The first such class is given by a certain subbicategory $\cH_A$ of modules over a finite-dimensional radically graded Hopf algebra (under tensor product over $\Bbbk$), whose $1$-morphisms are given by the additive closure of projective modules and the trivial module. The second example is the bicategory $\cG_A$ of projective symmetric bimodules over a finite-dimensional self-injective basic algebra $A$ with an action of a finite group $G$, generalising the $2$-categories studied in \cite{MMZ2}. We embed both bicategories into the bicategory $\cC_A$ of projective $A$-$A$-bimodules (with horizontal composition given by tensor product over $A$) in Corollary \ref{Gammares} resp.  Corollary \ref{thetares} and show that, for $A$ a finite-dimensional radically graded Hopf algebra, $\cH_A$ can be viewed as a $1$-full subbicategory of $\cG_A$ (see Theorem \ref{HAinGA}).
We show that $\cG_A$ (and hence, by $1$-fullness, $\cH_A$) has precisely two $\H$-cells $\{\one\}$ and $\H_0$ as desired in Proposition \ref{prop7}. Moreover, both bicategories are $\H_0$-simple (Proposition \ref{Hsimple}). Finally, we give a classification of simple transitive birepresentations (extending the case considered in \cite{MMZ2}) of $\cG_A$ in Theorem \ref{thm11}, of which there are only finitely many equivalence classes. By contrast, $\cH_A$ has infinitely many equivalence classes of simple transitive birepresentations in general.

{\bf Structure of the paper.} In Section \ref{prelims}, we recall the relevant notions from finitary birepresentation theory. In Section \ref{onefusion}, we prove that an $\H_0$-simple bicategory with a single $\H$-cell $\H_0$ is fusion. In Section \ref{Hopfsection}, we introduce the relevant bicategories of representations of a Hopf algebra $A$, and construct a pseudofunctor $\Gamma$ embedding these into the appropriate bicategories of $A$-$A$-bimodules. In Section \ref{s3}, we introduce the relevant bicategories $\cG_A$ of symmetric $A$-$A$-bimodules (generalising the setup from \cite{MMZ2}) and likewise embed these into the appropriate bicategories of $A$-$A$-bimodules via a pseudofunctor $\Theta$. Moreover, we verify that the essential image of $\cH_A$ under $\Gamma$ is a $1$-full subbicategory or the essential image of $\cG_A$ under $\Theta$. Finally, we classify simple transitive birepresentations of $\cG_A$.

{\bf Acknowledgements.} The research for this article was supported by EPSRC grant EP/S017216/1.

\section{Preliminaries}\label{prelims}
Here we recall some basic definitions from finitary birepresentation theory using the conventions from \cite{MMMTZ2}.

Throughout, we let $\Bbbk$ denote an algebraically closed field. We call a $\Bbbk$-linear, additive category finitary if it is idempotent complete, has finite-dimensional morphism spaces and only finitely many indecomposable objects up to isomorphism. We denote by $\mathfrak{A}_\Bbbk^f$ the $2$-category of finitary categories, $\Bbbk$-linear functors and natural transformations.

We call a bicategory $\cC$ multifinitary if the number of objects of $\cC$ is finite,
 the categories $\cC(\mathtt{i},\mathtt{j})$ are
finitary for all $\mathtt{i},\mathtt{j}\in\cC$, and horizontal composition of $2$-morphisms is $\Bbbk$-bilinear.
If, additionally, the identity $1$-morphism on each object is indecomposable, we say $\cC$ is finitary. 

A bicategory $\cC$ is quasi (multi)fiab if it is (multi)finitary and each $1$-morphism has a left and a right adjoint. If these are moreover isomorphic, $\cC$ is (multifiab), see \cite[Definition 2.5]{MMMTZ2} for more detail.

A finitary birepresentation $\bfM$ of a multifinitary bicategory $\cC$ is a pseudofunctor $\bfM\colon \cC\to \mathfrak{A}_\Bbbk^f$. A finitary birepresentation is called simple transitive if $\coprod_{\ti\in \cC}\bfM(\ti)$ has no proper $\cC$-stable ideals.

If $\cC$ is a bicategory, and $\ti\in \cC$, we denote by $\cC_\ti$ the endomorphism bicategory of $\ti$, more precisely, the bicategory with one object $\ti$ and morphism category $\cC(\ti,\ti)$.

The set of isomorphism classes  of indecomposable $1$-morphisms of a  multifinitary bicategory $\cC$ carries a left partial preorder $\leq_{L}$ generated by setting $\mathrm{F}\leq_{L}\mathrm{G}$ if there exists $\mathrm{H}$ such that $\mathrm{G}$ is isomorphic to a direct summand of $\mathrm{H}\mathrm{F}$. Similarly, one defines right and two-sided partial preorders $\leq_R$ and $\leq_J$, respectively. Equivalence classes with respect to these are called left, right and two-sided cells, respectively. Moreover, an $\H$-cell is an intersection of a left and a right cell.

By a mild generalisation of \cite[Subsection 3.2]{ChMa}, every finitary birepresentation $\bfM$ of a (multi)finitary bicategory has an apex, which is the unique maximal two-sided cell not annihilated by $\bfM$.
To each left cell $\L$, we can associate the cell birepresentation $\bfC_\L$, which is the quotient of the left $2$-ideal generated by the $1$-morphisms in $\L$ by its unique maximal $\cC$-stable ideal. By construction, this is simple transitive and its apex is the two-sided cell containing $\L$. For more details on cells and cell birepresentations, see \cite[Section 2.5]{MMMTZ2}.

If $\J$ is a two-sided cell in $\cC$, we say $\cC$ is $\J$-simple provided it has no proper $2$-ideals which do not contain the identities on $1$-morphisms in $\J$.

\section{A single $\H$-cell $\H_0$ and $\H_0$-simplicity implies fusion}\label{onefusion}

Throughout this section, let $\cC$ be a quasi fiab $2$-category with only one object $\bullet$ (hence denoting the identity $1$-morphism on $\bullet$ simply by $\one$), and with only one $\H$-cell $\H_0$.

\begin{proposition}
If $\cC$ is $\H_0$-simple, then $\cC$ is a fusion category.
\end{proposition}

\begin{proof}
Recall (cf. \cite[Definition 4.1.1]{EGNO}) that a bicategory $\cC$ is a fusion category if 

\begin{enumerate}
\item \label{object} it has only one object $\bullet$;
\item \label{morphcat}  the category $\cC(\bullet,\bullet)$ is a $\Bbbk$-linear, additive, with finitely many isomorphism classes of indecomposable objects;
\item \label{ss} $\cC(\bullet,\bullet)$ is semisimple and all objects have finite length;
\item \label{morspace} morphism spaces in $\cC(\bullet,\bullet)$ are finite-dimensional;
\item \label{horbilin} horizontal composition is bilinear on $2$-morphisms;
\item \label{indecomp} $\one=\one_\bullet$ is indecomposable;
\item \label{adj} it admits adjunctions (i.e. is rigid).
\end{enumerate}

By assumption, our $2$-category $\cC$ is quasi fiab on one object, which takes care of \eqref{object},\eqref{morphcat},\eqref{morspace}, \eqref{horbilin}, \eqref{indecomp} and \eqref{adj}, and provided that $\cC(\bullet,\bullet)$ is semisimple, all objects in $\cC(\bullet,\bullet)$ have finite length, 
so we only need to check that $\cC(\bullet,\bullet)$ is semisimple.

By \cite[Theorem 2]{KMMZ} applied to the case where $\bfM$ is the cell birepresentation $\bfC_{\H_0}$ associated to our (unique) cell ${\H_0}$, and $\rF=\one$, it follows that $\bfC_{\H_0}(\bullet)$ is semisimple.
Thus the kernel of the $2$-functor $\bfC_{\H_0}$ contains all $2$-morphisms which belong to the radical of $\bfC_{\H_0}(\bullet,\bullet)$. By ${\H_0}$-simplicity, the cell birepresentation $\bfC_{\H_0}$ is faithful, hence the radical of $\cC(\bullet,\bullet)$ is zero.
\end{proof}

Motivated by this result, the remainder of this article investigates certain classes of quasi fiab bicategories, which have precisely two $\H$-cells, namely one consisting of the identity $1$-morphism and precisely one other $\H$-cell (which is necessarily strictly larger in the two-sided order).

\section{Hopf algebras and projective bimodules}\label{Hopfsection}

Let $A$ be a finite dimensional unital associative algebra over $\Bbbk$. Additional assumptions on $A$ will be specified as we need them. Write $\oplus$ and $\otimes$ for the biproduct and tensor product on the category of $A$-$A$-bimodules. 

\subsection{Bimodule conventions} \label{bimodconv}

Let $K, M, N$ be $A$-$A$-bimodules. We fix the following standard canonical isomorphisms: 
$$ 
(K \otimes_A M) \otimes_A N \xrightarrow{\cong} K \otimes_A (M \otimes_A N),\ (k\otimes m) \otimes n \mapsto k \otimes (m \otimes n), $$
$$A \otimes_A M \xrightarrow{\cong} M,\  a \otimes m \mapsto am, $$
$$ K \otimes_A(M \oplus N) \xrightarrow{\cong} (K \otimes_A M) \oplus (K \otimes_A N),\ k \otimes (m,n) \mapsto (k \otimes m, k \otimes n).$$

Suppose $\{ e_1, \ldots, e_k\}$ is a complete set of primitive orthogonal idempotents for $A$. For all $i,j \in \{1, \ldots, k \}$ we have isomorphisms
$$ e_i A \otimes_A A e_j \xrightarrow{\cong} e_i A e_j, e_i a \otimes b e_j \mapsto e_i ab e_j.$$

Throughout, we treat all these isomorphisms as equalities.

\subsection{The bicategories $\cB_A$ and $\cC_A$}

Let $\cB_A$ be the bicategory of $A$-$A$-bimodules. More precisely, $\cB_A$ has one object $\bullet$ and its morphism category $\cB_A(\bullet,\bullet)$ is the category of $A$-$A$-bimodules with horizontal composition given by the tensor product $-\otimes_A -$, and the associator and unitors given by the canonical isomorphisms in Section \ref{bimodconv}.

We define $\cC_A$ to be the $2$-full subbicategory of $\cB_A$ whose $1$-morphisms are those bimodules in the additive closure of $A\oplus A\otimes_\Bbbk A$. Note that this is always multifinitary, quasi multifiab if $A$ is a Frobenius algebra, and multifiab if $A$ is weakly symmetric. The prefix multi- is superfluous if $A$ is indecomposable.

\begin{remark}
Note that if $A=A_1\times \cdots \times A_n$ is a decomposition of $A$ into indecomposable factors, then $\cC_A$ is a bicategorical version of the additive closure (see \cite[Section 2.4]{MMMTZ2}) of the $2$-category that is denoted by $\cC_A$ in e.g. \cite{MMZ3}.
\end{remark}

\subsection{Group actions and equivariant objects}\label{groupactions}


Let $G$ be a finite subgroup of the automorphism group of $A$, which we interpret as acting on the left of $A$.
We further assume that $\mathrm{char}(\Bbbk)$ does not divide $|G|$.
We obtain a right action of $G$ on the category of $A$-$A$-bimodules via
$M\mapsto M^g$, with the action of $A$ on $M^{g}$
given by
\begin{displaymath}
a\cdot m\cdot b:= g(a)m g(b), \quad \text{ for all }a,b\in A\text{ and }m\in M.
\end{displaymath}
For the translate of a morphism $\varphi$ under the action of
$g\in G$, we will write $\varphi^{g^{}}$.

\begin{lemma}
This is a well-defined action.
\end{lemma}

\proof
First note that $M^{g}$ is indeed an $A$-$A$-bimodule. This follows from the fact that $M$ is an $A$-$A$-bimodule, hence an $A\otimes_{\Bbbk}A^{op}$-module, and $g$ defines an algebra automorphism of $A\otimes_{\Bbbk}A^{op}$ for any $g \in G$. 

Clearly, $M \cong M^{1_{G}}$. Thus, we only need to check that $(M^{g})^h \cong M^{gh}$ as $A$-$A$-bimodules. We write $\cdot$ for the action on $M^{g}$ and $\ast$ for the action on $(M^{g})^h $.

Then the action of $A$ on the left hand-side is given by
$$ a \ast m \ast b=h(a)\cdot m\cdot h(b)=
g (h(a)) m g (h(b))=
 (gh)(a) m (gh)(b)
 $$
for $a,b \in A, \ m \in M$.
Thus, the action of $A$ is the same as required.
\endproof

\begin{lemma}\label{canoniso}
Let $M,N$ be two $A$-$A$-bimodules and $g\in G$. Then $(M\otimes_A N)^g\cong M^g\otimes_AN^g$ as $A$-$A$-bimodules via the canonical isomorphism $m\otimes n\mapsto m\otimes n$.
\end{lemma}

\proof
The map
$(M\otimes_A N)^g\to M^g\otimes_AN^g$, $m\otimes n\mapsto m\otimes n$ induces an isomorphism of $\Bbbk$-modules. This isomorphism extends to an isomorphism of $A$-$A$-bimodules since $g\in$ Aut$(A)$.
\endproof

For future use, we record the following lemma.

\begin{lemma}\label{actiononproj}
Assume that $A$ has a complete set of primitive orthogonal idempotents $\mathtt{E} = \{e_g | g \in G\}$ indexed by the group $G$, such that the action of $G$ leaves $\mathtt{E}$ invariant and is given by $h(e_g) = e_{gh^{\mone}}$.
Then $$(Ae_g\otimes_{\Bbbk}e_{g'}A)^h\cong Ae_{gh}\otimes_{\Bbbk}e_{g'h}A.$$
\end{lemma}

\proof
We compute
$e_s\cdot e_g\otimes e_{g'}\cdot e_t = h(e_s)e_g\otimes e_{g'}h(e_t) = e_{sh^{\mone}}e_g\otimes e_{g'}e_{th^{\mone}} \neq 0 $ iff $sh^{\mone}=g$ and $th^{\mone}=g'$, so
$(Ae_g\otimes_{\Bbbk}e_{g'}A)^h\cong Ae_{gh}\otimes_{\Bbbk}e_{g'h}A$ in the category of $A$-$A$-bimodules. 
\endproof

Recall (c.f. \cite{Su}) that a $G$-equivariant object of the category $\cB_A(\bullet, \bullet)$ is a pair $(M, \{ \alpha^M_g\ |\ g \in G \}$) where $\alpha_g^M: M \xrightarrow{\cong} M^g$ such that the diagram
\begin{equation}\label{G-equivobj}
\xymatrix{
M \ar^{\alpha_g^M}[r]\ar^{\alpha_{hg}^M}[d]& M^g\ar^{(\alpha_h^M)^{g}}[d]\\
M^{hg}\ar^{=}[r] & (M^h)^g
}
\end{equation}
commutes. By slight abuse of notation, we will call the object $M$ $G$-equivariant, if there exists a $G$-equivariant structure on it.

A morphism $\psi$ between two $G$-equivariant objects $(M, \{\alpha_g^M\ |\ g \in G\})$ and $(N, \{\alpha_g^N \ |\ g \in G\}$ is a morphism $\psi: M \to N$ such that the diagram

\begin{equation}\label{G-equivmor}
\xymatrix{
M \ar^{\alpha_g^M}[r]\ar^{\psi}[d]& M^g\ar^{(\psi)^g}[d]\\
N \ar^{\alpha_g^N}[r] & (N)^g.
}\end{equation}
commutes.

\subsection{Hopf algebra conventions}
Let $A$ be finite dimensional Hopf algebra over $\Bbbk$ with multiplication $\mathrm{m}$, comultiplication $\Delta$, counit $\epsilon$, unit $\iota$, and antipode $S$.

For $A$-modules $M, N$, we can, as usual, view $M\otimes_\Bbbk N$ as an A-module using the coproduct $\Delta$ of A via $a \cdot (m\otimes n)=\Delta(a)(m\otimes n)$. 

%
%

Let $B = A \otimes A^{op}$.
Consider the map $\varphi := (\id \otimes S) \circ \Delta: A \to B$. Since both $\Delta$ and $\id \otimes S$ are injective, so is $\varphi$.
Moreover, a direct computation shows that $\varphi$ is a morphism of algebras, and hence an algebra monomorphism. It thus induces a left and a right $A$-module structure on $B$ given by

\begin{equation}\begin{split}\label{AactiononB}
 a \cdot_\varphi (b \otimes c) &= \varphi(a)(b \otimes c)= \sum a^{(1)}b \otimes c S(a^{(2)}),\\ (b \otimes c) \cdot_\varphi a &= (b \otimes c) \varphi(a)= \sum ba^{(1)} \otimes S(a^{(2)})c,
\end{split}\end{equation}
respectively. When regarding $B$ as an $A$-$A$-bimodule with respect to these actions, we will write ${}_{A \cdot}B_{\cdot A}$.
For future use, we record the following lemma.

\begin{lemma}\label{AbasisofB}
Let $\mathtt{A}$ be a $\Bbbk$-basis of $A$. 
\begin{enumerate}[(a)]
\item\label{AbasisofB0} A $\Bbbk$-basis of $A\otimes_\Bbbk A$ is given by $\{\sum a^{(1)} \otimes a^{(2)}b\,\vert\, a, b\in \mathtt{A}\}$. In particular, $A\otimes_\Bbbk A$ as a left $A$-module with the action induced by the coproduct of $A$ is free on basis $\{\sum 1 \otimes b\,\vert\, b\in \mathtt{A}\}$.
\item\label{AbasisofB1} A $\Bbbk$-basis of $A\otimes_\Bbbk A$ is given by $\{\sum a^{(1)} \otimes S(a^{(2)})b\,\vert\, a, b\in \mathtt{A}\}$. In particular, $B_{\cdot A}$ is free as a right $A$-module with basis $\{\sum 1 \otimes b\,\vert\, b\in \mathtt{A}\}$.
\item\label{AbasisofB2} A $\Bbbk$-basis of $A\otimes_\Bbbk A$ is given by $\{\sum a^{(1)} \otimes bS(a^{(2)})\,\vert\, a, b\in \mathtt{A}\}$. In particular, ${}_{\cdot A}B$ is free as a left $A$-module with basis $\{\sum 1 \otimes b\,\vert\, b\in \mathtt{A}\}$
\end{enumerate}
\end{lemma}

\proof
Following the proof of \cite[Lemma 3.1.4]{Mo}, we define $f = (\id\otimes\mathrm{m})\circ(\Delta\otimes \id)\colon A\otimes A\to A\otimes A$. This has inverse $g = (\id\otimes \mathrm{m})(\id\otimes S\otimes \id)(\Delta\otimes \id)$, as commutativity of the diagram
$$
\xymatrix{A\otimes A \ar^{\Delta\otimes \id}[rr]\ar^{\Delta\otimes \id}[d]&&A\otimes A\otimes A \ar^{\id\otimes \mathrm{m}}[rr]\ar^{\Delta\otimes \id\otimes\id}[d]&& A\otimes A \ar^{\Delta\otimes \id}[d]\\
A\otimes A\otimes A\ar^{\id \otimes\Delta\otimes \id}[rr]
\ar_{\id\otimes \iota\epsilon\otimes \id}[rrdd]
 &&A\otimes A\otimes A \otimes A\ar^{\id \otimes\id\otimes \mathrm{m}}[rr] \ar^{\id\otimes S\otimes \id\otimes \id}[d] && A\otimes A\otimes A\ar^{\id\otimes S\otimes \id}[d]  \\
&&A\otimes A\otimes A \otimes A\ar^{\id \otimes\id\otimes \mathrm{m}}[rr] \ar^{\id \otimes \mathrm{m}\otimes\id}[d] && A\otimes A\otimes A \ar^{\id \otimes \mathrm{m}}[d] \\
&&A\otimes A\otimes A \ar^{\id \otimes \mathrm{m}}[rr] && A\otimes A,
}
$$
which follows from coassociativity and associativity for the top left and bottom right square, the interchange law for the the two top right squares, and the Hopf algebra axiom for the bottom left square,
shows. 

\eqref{AbasisofB0} Since $\{\sum a^{(1)} \otimes a^{(2)}b\,\vert\, a, b\in \mathtt{A}\}$ is the image of the $\Bbbk$-basis $\{a\otimes b \,\vert\, a,b \in \mathtt{A}\}$ under the isomorphism $f$ it is also a $\Bbbk$-basis.

\eqref{AbasisofB1} 
Similarly, $\{\sum a^{(1)} \otimes S(a^{(2)})b\,\vert\, a, b\in \mathtt{A}\}$ is the image of the $\Bbbk$-basis $\{a\otimes b \,\vert\, a,b \in \mathtt{A}\}$ under the isomorphism $g$, it is also a $\Bbbk$-basis, and \eqref{AbasisofB1} follows.

\eqref{AbasisofB2} Since $A$ is finite-dimensional , the antipode is invertible and  $\{a\otimes S^{\mone}(b) \,\vert\, a,b \in \mathtt{A}\}$ is also a $\Bbbk$-basis  of $A\otimes A$. The image of this basis under $f$ is $\{\sum a^{(1)} \otimes a^{(2)}S^{\mone}(b)\,\vert\,a,b\in \mathtt{A}\}$, which under the isomorphism $\id\otimes S$ is sent to $\{\sum a^{(1)} \otimes bS(a^{(2)})\,\vert\,a,b\in \mathtt{A}\}$, proving that the latter is also a $\Bbbk$-basis.

All statements about $A$-bases with respect to the given actions follow immediately from the definitions of the latter (see \eqref{AactiononB} for \eqref{AbasisofB1},\eqref{AbasisofB2}).
\endproof

\subsection{The bicategories  $\cR ep_A$ and $\cH_A$.}

Consider $\cR ep_A$, the one object bicategory with object $\bullet$ and $\cR ep_A(\bullet, \bullet)\cong A\lmod$, with a fixed biproduct $-\oplus - $ and with horizontal composition induced by $-\otimes_\Bbbk -$ and the Hopf algebra structure on $A$. 
The associator is given by the canonical isomorphism $(K\otimes_\Bbbk M)\otimes_\Bbbk N \xrightarrow{\id\otimes \id\otimes \id}K\otimes_\Bbbk (M\otimes_\Bbbk N)$, which we again treat as an identity, and the unit object is given by $L_1 = \Bbbk v$, on which $A$ acts by the counit, i.e. $av=\epsilon(a)v$. The unitors are then given by $L_1 \otimes_\Bbbk M \to M: v\otimes m\mapsto m$ and $M\otimes_\Bbbk L_1  \to M: m\otimes v\mapsto m$.

We define $\cH_A$ to be the $2$-full subbicategory whose $1$-morphisms are those $A$-modules in the additive closure of $A\oplus L_1$. This is finitary and, given a finite-dimensional Hopf algebra is always Frobenius, also quasi fiab.


\subsection{The pseudofunctor $\Gamma$}\label{Gamma}

We will now define a $2$-faithful pseudofunctor $\Gamma\colon \cR ep_A\to \cB_A$.  

Recall $B=A\otimes_\Bbbk A^{op}$ and the left and right $A$-module structure on $B$ induced by $\varphi$, see \eqref{AactiononB}. Identifying $ \cB_A(\bullet,\bullet)$ with the $B\lmod$,  consider the restriction functor $\Phi: \cB_A(\bullet,\bullet) \to \cR ep_A(\bullet,\bullet)$, which can equivalently be viewed as $\Hom_B(B_{\cdot A}, -)$ or as ${}_{A\cdot}B\otimes_B-$. It therefore has a left adjoint $\Gamma$ given by $B_{\cdot A}\otimes_A -$.

\begin{lemma}\label{Philax}
$\Phi$ is a lax pseudofunctor.
\end{lemma}
\begin{proof}
Let $M,N\in B\lmod$. We define a natural transformation $$\kappa\colon \Phi(-)\otimes_\Bbbk \Phi(-) \to \Phi(-\otimes_A -)$$ by the natural projection  
$$\kappa_{M,N}\colon M\otimes_\Bbbk N \twoheadrightarrow M\otimes_AN.$$
We note that this is indeed a morphism in $\cR ep_A(\bullet,\bullet)$, since
\begin{equation*}\begin{split}
\kappa_{M,N} (a\cdot_\varphi (m\otimes_\Bbbk n)) 
&= \kappa_{M,N} ((a^{(1)}\cdot_\varphi m)\otimes_\Bbbk (a^{(2)}\cdot_\varphi n)) \\
&= \kappa_{M,N} (a^{(11)}mS(a^{(12)})\otimes_\Bbbk a^{(21)}nS(a^{(22)})) \\
&=a^{(11)}mS(a^{(12)})\otimes_A a^{(21)}nS(a^{(22)})\\
&=a^{(1)}m\otimes_A nS(a^{(2)}) = a\cdot_\varphi (m\otimes_A n).
\end{split}\end{equation*}
Here the fourth equality follows from commutativity of
$$
\xymatrix{
A \ar^\Delta[rr] \ar^\Delta[d]&& A\otimes_\Bbbk A  \ar^{\Delta\otimes \id}[d]
\\
 A\otimes_\Bbbk A\ar^{\id\otimes \Delta}[rr] \ar^{\id\otimes \Delta}[d] && A\otimes_\Bbbk A\otimes_\Bbbk A\ar^{\id\otimes \id\otimes \Delta}[d] 
 \\
 A\otimes_\Bbbk A\otimes_\Bbbk A \ar^{\id\otimes \Delta\otimes \id}[rr]\ar^{\id \otimes \epsilon \otimes \id}[d]&& A\otimes_\Bbbk A\otimes_\Bbbk A\otimes_\Bbbk A \ar^{\id\otimes S\otimes  \id\otimes S}[d]\\
 A\otimes_\Bbbk\Bbbk\otimes_\Bbbk A \ar^{\id \otimes \iota \otimes S}[dr]&&A\otimes_\Bbbk A\otimes_\Bbbk A\otimes_\Bbbk A\ar^{\id\otimes \mathrm{m}\otimes \id}[dl]\\
 & A\otimes_\Bbbk A\otimes_\Bbbk A&
}$$
which is due to coassociativity for the top two squares and the antipode  axiom for the bottom pentagon.
Compatibility with the associator is encoded by the commutative diagram
$$
\xymatrix{
(K\otimes_\Bbbk M)\otimes_\Bbbk N \ar^{\id\otimes \id\otimes \id}[rr]\ar^{\kappa_{K,M}\otimes \id}[d]&&K\otimes_\Bbbk (M\otimes_\Bbbk N)\ar^{\id \otimes \kappa_{M,N}}[d] \\
(K\otimes_A M)\otimes_\Bbbk N \ar^{\kappa_{K\otimes_A M,N}}[d]&&K\otimes_\Bbbk (M\otimes_A N)\ar^{\kappa_{K, M\otimes_A N}}[d] \\
(K\otimes_A M)\otimes_A N  \ar^{\id\otimes \id\otimes \id}[rr]&&K\otimes_A(M\otimes_A N).
}
$$
Finally, consider the unit $L_1$, and define the morphism $\xi\colon L_1\to \Phi(A)$ by $v\mapsto 1_A$. This is indeed a morphism in $\cR ep(\bullet,\bullet)$, since $\xi (av) = \xi(\epsilon(a)v) = \epsilon(a)1_A$ while $a\cdot_\varphi 1_A = \sum a^{(1)}1_AS(a^{(2)}) = \epsilon(a)1_A$ by the antipode axiom. Moreover, the diagrams
$$
\xymatrix{
L_1\otimes_\Bbbk M \ar^{\xi\otimes \id}[r]\ar^{\sim}[d]& A\otimes_\Bbbk M\ar^{\kappa_{A,M}}[d]\\
M \ar^{\sim}[r] & A\otimes_AM
}
\qquad\qquad
\xymatrix{
M\otimes_\Bbbk L_1 \ar^{\id\otimes \xi}[r]\ar^{\sim}[d]& M\otimes_\Bbbk A\ar^{\kappa_{M,A}}[d]\\
M \ar^{\sim}[r] & M\otimes_AA
}
$$
commute, hence $\kappa$ and $\xi$ equip $\Phi$ with the structure of a lax pseudofunctor.
\end{proof}

Since $\Phi$ is isomorphic to $\Hom_B(B_{\cdot A}, -)$, it has a left adjoint $\Gamma$ given by $B_{\cdot A}\otimes_A -$. Note that this is an exact functor by Lemma \ref{AbasisofB}\eqref{AbasisofB1}. Denote by $\sigma\colon \Id\to \Phi\Gamma$ and $\tau \colon \Gamma\Phi\to \Id$ the unit and counit of the adjunction $(\Gamma, \Phi)$.

\begin{corollary}\label{Gammaoplax}
The left adjoint $\Gamma \colon \cR ep_A\to \cB_A$ of $\Phi$ is an oplax pseudofunctor, with the corresponding natural transformation $\gamma\colon \Gamma(- \otimes_{\Bbbk} -) \to \Gamma(-) \otimes_A \Gamma(-)$ given by the composite
$$ 
\xymatrix{\gamma_{X,Y}\,\colon \,
\Gamma(X \otimes_{\Bbbk} Y) \ar^{\Gamma(\sigma_X\otimes \sigma_Y)}[rr] && \Gamma( \Phi\Gamma(X) \otimes_{\Bbbk}  \Phi\Gamma(Y))\ar^{\Gamma(\kappa_{\Gamma(X),\Gamma(Y)})}[rr]&& \Gamma\Phi(\Gamma(X)\otimes_A\Gamma(Y))\ar_{\tau_{\Gamma(X)\otimes_A\Gamma(Y)}} [d]\\ &&&&\Gamma(X) \otimes_A \Gamma(Y)
} $$
and the morphism $\zeta\colon \Gamma(L_1)\to A$ given as the image of $\xi$ under the adjunction isomorphism
$$\Hom_A(L_1,\Phi(A)) \cong \Hom_{A-A}(\Gamma(L_1), A).$$
\end{corollary}

Our goal is to show that $\Gamma$ is indeed a (strong) pseudofunctor, meaning $\gamma$ and $\zeta$ are isomorphims.

\begin{lemma} \label{zetaiso}
The morphism $\zeta$ defined in Corollary \ref{Gammaoplax} is an isomorphism.
\end{lemma}
\proof
Recall that $\Phi(A)=A = \Hom_B(B_{\cdot A}, A)$ and $\zeta\colon \Gamma(L_1)\to A$ is defined as the image of $\xi: L_1 \to \Phi(A)$, as defined in the proof of Lemma \ref{Philax}, under the adjunction isomorphism
$$\Hom_A(L_1,\Phi(A)) \cong \Hom_{A-A}(\Gamma(L_1), A).$$
Identifying $\Phi(A)$ with $\Hom_B(B_{\cdot A}, A)$, we see that $\xi(v): B_{\cdot A} \to A$ is the unique map of $A$-$A$-bimodules sending $ 1 \otimes 1 \mapsto 1$.

By Lemma \ref{AbasisofB}\eqref{AbasisofB1} and the action of $A$ on $L_1$, a $\Bbbk$-basis of $\Gamma(L_1)$ is given by $1 \otimes b \otimes v$, where $b\in\mathtt{A}$, for a $\Bbbk$-basis $\mathtt{A}$ of $A$. Pulling $\xi$ through the adjunction, we see that 
$\zeta(1 \otimes b \otimes v)= \xi(v)(1\otimes b) = b$, thus $\zeta$ produces a bijection of $\Bbbk$-bases and is hence an isomorphism, as claimed.
\endproof

In order to do show that $\gamma$ is an isomorphism, we first determine the adjunction morphisms explicitly.

\begin{lemma}
The unit $\sigma\colon \Id\to \Phi\Gamma$ and counit $\tau \colon \Gamma\Phi\to \Id$ of the adjunction $(\Gamma, \Phi)$ are given by
$\varphi\otimes_A-$ and $\tilde{\tau}\otimes_{A\otimes A^{op}}-$, respectively, where
$$\tilde{\tau}\colon  B_{\cdot A}\otimes_A {}_{A\cdot}B \to B,  (a\otimes b)\otimes(c\otimes d) \mapsto ac\otimes db.$$
\end{lemma}

\proof
In order to check that the compositions 
$$\xymatrix{\Phi\ar[r]^(.4){\sigma\circ_0\id}& \Phi\Gamma\Phi\ar[r]^(.6){\id\circ_0\tau}& \Phi}\qquad \text{and}\qquad \xymatrix{\Gamma\ar[r]^(.4){\id\circ_0\sigma}&\Gamma\Phi\Gamma\ar[r]^(.6){\tau\circ_0\id}& \Gamma}$$ are the respective identities, it suffices to check the representing maps on bimodules. 

Note that the natural isomorphism $A\otimes_A{}_{\cdot A}B \cong {}_{\cdot A}B$ identifies $a\otimes_A (1 \otimes b)$ with $ a \cdot_{\varphi} (1 \otimes b) = \sum a^{(1)} \otimes b S(a^{(2)})$, thus the composition
$$\begin{array}{ccccc}
A\otimes_A{}_{\cdot A}B&\xrightarrow{\tilde\sigma\otimes\id}&{}_{\cdot A}B_{\cdot A}\otimes_A {}_{A\cdot}B&\xrightarrow{\id \otimes \tilde\tau}&{}_{\cdot A}B,\\
a\otimes_A (1 \otimes b) &\longmapsto&\sum a^{(1)}\otimes S(a^{(2)}) \otimes 1 \otimes b &\longmapsto& \sum a^{(1)} \otimes b S(a^{(2)})
\end{array}$$
is indeed the identity.

Similarly, since the natural isomorphism $B_{\cdot A} \otimes_A A$ identifies $(1 \otimes b)\otimes a$ with $(1 \otimes b) \cdot_{\varphi} a=\sum a^{(1)} \otimes S(a^{(2)})b$, 
the second composition given by
$$\begin{array}{ccccc}
B_{\cdot A}\otimes_A A& \xrightarrow{\id\circ_0\otimes \tilde\sigma} &B_{\cdot A} \otimes_{\cdot A} B\otimes_B B_{\cdot A} &\xrightarrow{\tilde\tau\otimes\id} &B_{\cdot A}\\
(1\otimes b)\otimes_A a &\longmapsto &1 \otimes b \otimes_A \sum a^{(1)}\otimes S(a^{(2)})& \longmapsto& \sum a^{(1)} \otimes S(a^{(2)})b
\end{array}$$
is also the identity.
\endproof

\begin{lemma}\label{AAimpliesXY}
If $\gamma_{A,A}$ is an isomorphism, then $\gamma_{X,Y}$ is an isomorphism for all $X,Y\in \cR ep_A (\bullet,\bullet)$.
\end{lemma}

\begin{proof}
Let $X, Y$ be objects in $\cR ep_A (\bullet,\bullet)$ with free presentations $F_1 \xrightarrow{f} F_0 \twoheadrightarrow X$ and $G_1 \xrightarrow{g} G_0 \twoheadrightarrow Y$ in $\cR ep_A (\bullet,\bullet)$. Since the bifunctor $ - \otimes_{\Bbbk} -$ giving the monoidal structure on $\cR ep_A (\bullet,\bullet)$ is exact, we obtain that 
$$
\xymatrix{
F_1 \otimes_{\Bbbk} G_0 \oplus F_0 \otimes_{\Bbbk} G_1 \ar^(.6){(f \otimes \id, \id \otimes g)}[rr]
&& F_0 \otimes_{\Bbbk} G_0 \ar@{->>}[r]
& X \otimes_{\Bbbk} Y
} 
 $$
is a free presentation of $X \otimes_{\Bbbk} Y$ in $\cR ep_A (\bullet,\bullet)$.

Since $\Gamma(A)=B$ and $\Gamma$ is exact
$$\xymatrix{
\Gamma(F_1 \otimes_{\Bbbk} G_0 \oplus F_0 \otimes_{\Bbbk} G_1) \ar[rr]^(.6){\Gamma(f \otimes \id, \id \otimes g)} 
&&\Gamma(F_0 \otimes_{\Bbbk} G_0) \ar@{->>}[r] & \Gamma(X \otimes_{\Bbbk} Y)
}
$$
is a free presentation of $\Gamma(X \otimes_{\Bbbk} Y)$ in $\cB _A (\bullet,\bullet)$.
Thus, we have a diagram
$$
\xymatrix{
\Gamma(F_1 \otimes_{\Bbbk} G_0) \oplus \Gamma(F_0 \otimes_{\Bbbk} G_1) \ar^(.6){ (\Gamma(f \otimes \id), \Gamma(\id \otimes g))}[rr] \ar^{\left( \begin{smallmatrix} \gamma_{F_{1},G_{0}}& 0\\ 0& \gamma_{F_{0},G_{1}}\end{smallmatrix} \right)}[d] 
&& \Gamma(F_0 \otimes_{\Bbbk} G_0) \ar@{->>}[r] \ar^{\gamma_{F_{0},G_{0}}}[d]
& \Gamma(X \otimes_{\Bbbk} Y) \ar@{-->}[d]^{\gamma_{X,Y}} \\
\Gamma(F_1) \otimes_A \Gamma(G_0) \oplus \Gamma(F_0) \otimes_A \Gamma (G_1) \ar^(.625){ (\Gamma(f) \otimes \id, \id\otimes\Gamma( g))}[rr] 
&& \Gamma(F_0) \otimes_A \Gamma(G_0) \ar@{->>}[r] 
&  \Gamma(X) \otimes_A \Gamma(Y) 
}.
$$
Since the $\Gamma(F_i),\Gamma(G_i), i=0,1,$ are free $B$- and in particular free $A$-modules, the bottom row is a free presentation of $\Gamma(X) \otimes_A \Gamma(Y) $. By naturality of $\gamma$, the induced cokernel map is $\gamma_{X,Y}$. As $F_0, F_1, G_0$ and $G_1$ are free $A$-modules, the maps $\gamma_{F_{i}, G_{j}}$, where $i,j \in \{0,1\}$, are direct sums of copies of $\gamma_{A,A}$ and are thus isomorphisms. Since the two vertical maps in the diagram above are isomorphisms, it follows that so is $\gamma_{X,Y}$.
\end{proof}

In order to prove that $\gamma_{A,A}$ is an isomorphism, we now provide basis for its domain and codomain. To this end, let again $\mathtt{A}$ be a $\Bbbk$-basis of $A$.

\begin{lemma} \label{Yabc}
For $a,b,c\in A$, set $Y_{abc} := \sum a^{(1)} \otimes_{\Bbbk} S(a^{(2)})b \otimes_A 1 \otimes_{\Bbbk} c$. Then
$\{ Y_{abc} \ | \ a,b,c \in \mathtt{A} \}$ is a basis for $\Gamma(A\otimes_\Bbbk A)$.
\end{lemma}
\begin{proof}
Note that $\Gamma(A\otimes_\Bbbk A) =B_{\cdot A} \otimes_A (A \otimes_{\Bbbk} A)$, where the left action of $A$ on $A\otimes_\Bbbk A$ is induced by the coproduct. By Lemma \ref{AbasisofB}, both the left and the right side of the tensor product are free as right, respectively left, $A$-modules on bases $\{1 \otimes b \ \vert \ b\in \mathtt{A}  \}$, respectively $\{1 \otimes c \ \vert  \ c\in \mathtt{A}  \}$. Thus 
$B_{\cdot A} \otimes_A (A \otimes_{\Bbbk} A) =  B_{\cdot A} \otimes_A A\otimes _A (A \otimes_{\Bbbk} A )$ (where we again treat the canonical isomorphism as an identity) has $\Bbbk$-basis $\{1 \otimes b \otimes_A a \otimes_A 1 \otimes_{\Bbbk} c \ \vert \ a,b,c \in \mathtt{A}\}$, and using the the definition of the action in \eqref{AactiononB}, we see that $1 \otimes b \otimes_A a \otimes_A 1 \otimes_{\Bbbk} c = Y_{a,b,c}$,  completing the proof.
\end{proof}

\begin{lemma} \label{Xabc} For $a,b,c \in A$, set $X_{abc} := \sum a^{(1)} \otimes c^{(1)} \otimes S(a^{(2)}c^{(2)})b$. Then
$\{ X_{abc}  \ | \ a,b,c \in \mathtt{A}\}$ is a $\Bbbk$-basis of $A \otimes_{\Bbbk} A \otimes_{\Bbbk} A$.
\end{lemma}

\begin{proof}
Consider the $\Bbbk$-basis  of $A \otimes_{\Bbbk} A \otimes_{\Bbbk} A$ given by $\{ a \otimes c \otimes b\ |\ a,b,c \in\mathtt{A} \}$.

Define the maps
$$\varphi: A \otimes_{\Bbbk} A \otimes_{\Bbbk} A \to A \otimes_{\Bbbk} A \otimes_{\Bbbk} A, \quad a \otimes_{\Bbbk} c \otimes_{\Bbbk} b \mapsto \sum a^{(1)} \otimes_{\Bbbk} c^{(1)} \otimes_{\Bbbk} S(a^{(2)}c^{(2)})b $$
and
$$\psi: A \otimes_{\Bbbk} A \otimes_{\Bbbk} A \to A \otimes_{\Bbbk} A \otimes_{\Bbbk} A, \quad  a \otimes_{\Bbbk} c \otimes_{\Bbbk} b \mapsto \sum a^{(1)} \otimes_{\Bbbk} c^{(1)} \otimes_{\Bbbk} a^{(2)}c^{(2)}b.$$

To check that $\varphi$ and $\psi$ are inverse to each other, we compute
\begin{equation*}\begin{split}
\psi(\varphi(a \otimes c \otimes b))
&=\psi( \sum a^{(1)} \otimes c^{(1)} \otimes S(a^{(2)}c^{(2)})b)\\
&= \sum a^{(1)} \otimes c^{(1)} \otimes a^{(2)}c^{(2)}S(a^{(3)}c^{(3)})b
\end{split}
\end{equation*}
where we have used coassociativity of the comultiplication in the indexing. Using the Hopf algebra axiom, $\sum a^{(1)} \otimes c^{(1)} \otimes a^{(2)}c^{(2)}S(a^{(3)}c^{(3)}) = a\otimes c \otimes 1$, and hence $\psi(\varphi(a \otimes c \otimes b)) = a \otimes c \otimes b$. Thus $\{ X_{abc}  \ | \ a,b,c \in \mathtt{A}\}$ is the image of a $\Bbbk$-basis of $A \otimes_{\Bbbk} A \otimes_{\Bbbk} A$ under an isomorphism, hence also a $\Bbbk$-basis.
\end{proof}

\begin{lemma}\label{gammaiso}
The map $\gamma_{A,A}$ is an isomorphism.
\end{lemma}

\proof
Recall from Corollary \ref{Gammaoplax} that $$\gamma_{A,A} = \tau_{\Gamma A \otimes_A \Gamma A}\circ \Gamma(\kappa_{\Gamma A, \Gamma A} \circ (\sigma_A \otimes \sigma_A))\ \colon \,\,\Gamma(A\otimes_\Bbbk A) \to \Gamma(A)\otimes_A\Gamma(A).$$ By Lemma \ref{Yabc}, the elements $Y_{abc}=\sum a^{(1)} \otimes_{\Bbbk} S(a^{(2)})b \otimes_A 1 \otimes_{\Bbbk} c$, for $a,b,c\in \mathtt{A}$, form a basis of the domain.


Note that as a vector space $\Gamma(A)\otimes_A \Gamma(A) = (A\otimes_\Bbbk A)\otimes _A(A\otimes_\Bbbk A)$, where the $A$ action on the left tensor factor $A\otimes_\Bbbk A$ is given by \eqref{AactiononB}, and the action on the right tensor factor $A\otimes_\Bbbk A$ is just the left action on its left tensor factor $A$. It follows that $\Gamma(A)\otimes_A \Gamma(A) \cong A\otimes_\Bbbk A\otimes_\Bbbk A)$ as a vector space, and a $\Bbbk$-basis is given by $\{ X_{abc}  \ | \ a,b,c \in \mathtt{A}\}$.

We claim that $\gamma_{A,A}(Y_{abc}) = X_{abc}$. 

Indeed, first identifying $Y_{abc}$ with $a^{(1)} \otimes_{\Bbbk} S(a^{(2)})b \otimes_A 1 \otimes c \in \Gamma(A\otimes_\Bbbk A) = B_{\cdot A}\otimes_A(A\otimes_\Bbbk A) $ compute
\begin{equation*}
\begin{split}
\gamma_{A,A}(Y_{abc}) &= \tau_{\Gamma A \otimes_A \Gamma A}\circ \Gamma(\kappa_{\Gamma A, \Gamma A} \circ (\sigma_A \otimes \sigma_A)) ((a^{(1)} \otimes S(a^{(2)})b )\otimes_A (1 \otimes c))\\
&=  \tau_{\Gamma A \otimes_A \Gamma A}\circ \Gamma(\kappa_{\Gamma A, \Gamma A} )   ((a^{(1)} \otimes S(a^{(2)})b )\otimes_A (1 \otimes 1 \otimes c^{(1)} \otimes S(c^{(2)})))\\
&=\tau_{\Gamma A \otimes_A \Gamma A}((a^{(1)} \otimes S(a^{(2)})b) \otimes_A (1\otimes c^{(1)} \otimes S(c^{(2)})))\\
&=a^{(1)} \otimes c^{(1)} \otimes S(c^{(2)})S(a^{(2)})b\\
&=a^{(1)} \otimes c^{(1)} \otimes S(a^{(2)}c^{(2)})b = X_{abc},
\end{split}
\end{equation*}
proving the claim. Thus, $\gamma_{A,A}$ is bijective, and the statement follows.
%
%
\endproof

\begin{proposition}\label{Gammapseudo}
$\Gamma$ is a strong pseudofunctor.
\end{proposition}

\proof
This follows from Lemmas \ref{zetaiso}, \ref{AAimpliesXY} and \ref{gammaiso}
\endproof

\begin{corollary}\label{Gammares}
The pseudofunctor $\Gamma$ restricts to a pseudofunctor from $\cH_A$ to $\cC_A$.
\end{corollary}

\proof
This follows immediately from the fact that $\Gamma(A) \cong A\otimes_\Bbbk A$ and $\Gamma(L_1)\cong A$.
\endproof

\subsection{Finite-dimensional radically graded basic Hopf algebras}\label{radgrbasichopf}
Let $G$ be a finite group. Let  $W=(w_1, \ldots, w_n)$  be a weight sequence of elements in $G$, i.e. a sequence invariant under conjugation up to permutation. Following \cite{GS}, we can associate a quiver $Q=Q_{G,W}$, called the covering quiver, to the pair $(G,W)$ as follows: its vertices are labelled by elements of $G$, i.e. $Q_0=\{ e_g \ \vert \ g \in G\}$ and its arrows are given by
$$Q_1= \{a_{i,g} : e_{g^{\mone}} \to e_{w_i g^{\mone}} \ |\ i=1,\ldots, n,\ g \in G\}.$$
The path algebra $\Bbbk Q$ is said to have an allowable $\Bbbk G$-bimodule structure if it has a $\Bbbk G$-bimodule structure satisfying $g\cdot e_h \cdot g'= e_{g'^{\mone}hg^{\mone}}$, and such that $h \cdot a_{i,g} \cdot h'$ is contained in the $\Bbbk$-linear span of arrows from $e_{h'^{\mone}g^{\mone}h^{\mone}}$ to  $e_{h'^{\mone}w_i g^{\mone}h^{\mone}} $.

Let $A$ be a radically graded basic Hopf algebra. Then, by \cite[Theorem 2.1]{GS} and \cite[Lemma 2.5]{HL}, there exists a pair $(G,W)$ with associated covering quiver $Q$ and allowable $\Bbbk G$-bimodule structure as above, such that $A\cong \Bbbk Q/I$ for an admissible Hopf ideal $I$.
Counit, antipode and comultiplication are then defined on $Q_0$ and $Q_1$ by
$$\begin{array}{rclcrcl}\epsilon(e_g)&=&\begin{cases}1, & \text{for}\ g=1_G \\ 0 & \text{otherwise} \end{cases}
&& \epsilon(a_{i,g})&=&0 \\
S(e_g)&=&e_{g^{\mone}} &&  S(a_{i,g})&=& - w_i g^{\mone}  \cdot a_{i,g}\cdot   g^{\mone}\\
\Delta(e_g)&=&\sum_{h \in G} e_{gh} \otimes e_{h^{\mone}} && \Delta(a_{i,g})&=&\sum_{h \in G} (h \cdot a_{i,g} \otimes e_h + e_h \otimes a_{i,g} \cdot h)
\end{array}$$
and extended linearly and multiplicatively from there. Note that $S(a_{i,g}) = e_{g}S(a_{i,g}) e_{gw_i^{\mone}}$.

For the rest of this section, let $A$ be a radically graded basic Hopf algebra given by the data above.

Note that the left action of $G$ on $A$ induces a right action of $A$ on $\cB_A(\bullet,\bullet)$ (see Section \ref{groupactions}, and that, in particular, we are in the situation of Lemma \ref{actiononproj}, i.e $(Ae_g\otimes e_{g'}A)^h\cong Ae_{gh}\otimes e_{g'h}A)$.


\begin{lemma}\label{Gamma-proj-equiv}
For any $g\in G$, the $A$-$A$-bimodule $\Gamma(Ae_g)$ can be equipped with a $G$-equivariant structure. Moreover, for any morphism $\rho\colon Ae_g \to Ae_{g'}$ of left $A$-modules, $\Gamma(\rho)$ is a $G$-equivariant morphism.
\end{lemma}

\begin{proof}
We first construct the isomorphism $\alpha^{\Gamma(Ae_g)}_k\colon \Gamma(Ae_g)\to \Gamma(Ae_g)^k$, writing $\alpha_k$ for simplicity. Computing
$$\varphi(e_g) = (\id\otimes S)\left( \sum_{h\in G} e_{gh} \otimes e_{h^{\mone}}\right) = \sum_{h\in G} e_{gh}\otimes e_{h},$$
we see that 
$\Gamma(Ae_g) =(A\otimes A) \cdot_{\varphi} e_g = \bigoplus_{h\in H}A  e_{gh}\otimes e_{h} A$ and thus
 $$\Gamma(Ae_g)^k = (\bigoplus_{h \in G} Ae_{gh} \otimes e_h A)^k \cong \bigoplus_{h \in G} Ae_{ghk} \otimes e_{hk} A.$$
The obvious isomorphim is thus given by $\alpha_k(e_{gh}\otimes e_h) = e_{gh}\otimes e_h = e_{g \tilde{h}k} \otimes e_{\tilde{h} k}$ for $\tilde{h}=hk^{\mone}$, that is, the identical idempotent, which now lives in the $\tilde{h}$ component of $\Gamma(Ae_g)^k$. Given that $\alpha_k$ is just a relabeling but the underlying map is indeed an identity morphism, it is obvious that the diagram in \eqref{G-equivobj} commutes and this indeed defines a $G$-equivariant structure on $\Gamma(Ae_g)$.

Let now $\Gamma(Ae_g)$ and $\Gamma(Ae_{g'})$, $g,g' \in G$ and let $\rho_a: \Gamma(Ae_g) \to \Gamma(Ae_{g'})$ be defined by $e_g \mapsto e_{g'}a$, where $a=e_g a e_{g'}$.
Therefore, $\Gamma(\rho_a): \bigoplus_{h \in G} Ae_{gh} \otimes e_h A \to \bigoplus_{h' \in G} Ae_{g'h'} \otimes e_{h'} A$ is defined by $e_{gh} \otimes e_{h} \mapsto \sum a^{(1)} \otimes S(a^{(2)})$. We need to show that $\Gamma(\rho_a)$ is a $G$-equivariant morphism (see \eqref{G-equivmor}), i.e. that the diagram 
\begin{equation*}
\xymatrix{
\Gamma(Ae_g) \ar^{\alpha_k}[r]\ar^{\Gamma(\rho_a)}[d]& \Gamma(Ae_g)^k\ar^{\Gamma(\rho_a)^g}[d]\\
\Gamma(Ae_{g'})\ar^{\alpha_k}[r] & \Gamma(Ae_{g'})^k
}
\end{equation*}
commutes.

Again, we identify $\Gamma(Ae_g)^k$ with $$\bigoplus_{h \in G} Ae_{ghk} \otimes e_{hk} A= \bigoplus_{\tilde h \in G} Ae_{g\tilde h} \otimes e_{\tilde h} A$$ for  and $\Gamma(Ae_{g'})^k$ with  $$\bigoplus_{h' \in G} Ae_{g'h'k} \otimes e_{h'k} A = \bigoplus_{\tilde h' \in G} Ae_{g\tilde h'} \otimes e_{\tilde h'} A$$ for $\tilde{h}=hk^{\mone}$ and $\tilde{h}'=h'k^{\mone}$.
We compute the $\tilde{h}'$ components of both compositions applied to the generator of the $h$-component $e_{gh}\otimes e_h$. 

On the one hand, we obtain $$\left(\alpha_k(\Gamma(\rho_a) (e_{gh}\otimes e_h)\right)_{\tilde{h}'}=\sum e_{gh} a^{(1)} e_{g'\tilde{h}'} \otimes e_{\tilde{h}'} S(a^{(2)})e_h.$$

On the other hand, we compute 
\begin{equation*}
\begin{split}
\left(\Gamma(\rho_a)^k(\alpha_k(e_{gh}\otimes e_h))\right)_{\tilde{h}'}
&=\left(\Gamma(\rho_a)^k(e_{g \tilde{h}k} \otimes e_{\tilde{h} k}) \right)_{\tilde{h}'}\\
&= \left(k^{\mone}\cdot(\Gamma(\rho_a)(e_{g \tilde{h}} \otimes e_{\tilde{h}})) \right)_{\tilde{h}'}\\
&=\left(k^{\mone}\cdot( \sum e_{g \tilde{h}} a^{(1)} \otimes S(a^{(2)}) e_{\tilde{h}}) \right)_{\tilde{h}'}\\
&= \sum e_{g h} (k^{\mone}\cdot a^{(1)}) e_{g'\tilde{h}'}\otimes e_{\tilde{h}'} (k^{\mone}\cdot S(a^{(2)})) e_{h}\\
\end{split}
\end{equation*}
where we have used that $k^{\mone}\cdot e_{g \tilde{h}} = e_{g h}$.

It thus suffices to verify that $ \sum  k\cdot a^{(1)} \otimes k\cdot S(a^{(2)}) = \sum a^{(1)} \otimes S(a^{(2)})$ for all $k \in G$, and it suffices to do this for arrows in the quiver of $A$.

 Let $a_{i, g}$ be an arrow. Then
\begin{equation*}
\begin{split}
(\id \otimes S) \Delta(a_{i,g})&=(\id \otimes S) (\sum_{h\in G} (h \cdot a_{i,g} \otimes e_h + e_h \otimes a_{i,g} \cdot h)) \\
&=\sum_{h\in G}(h \cdot a_{i,g}  \otimes e_{h^{\mone}} + e_h \otimes h^{\mone} \cdot S(a_{i,g}) ) \\
& = \sum a^{(1)} \otimes S(a^{(2)}).
 \end{split}
\end{equation*}
where we have used that $S(x\cdot h) = h^{\mone}\cdot S(x)$ (see \cite[Lemma 2.2]{GS}).
On the other hand,
\begin{equation*}
\begin{split}
\sum k \cdot a^{(1)} \otimes k \cdot S(a^{(2)}) 
&= \sum_{h \in G}
(kh \cdot a_{i,g}  \otimes e_{h^{\mone}k^{\mone}} + e_{hk^{\mone}} \otimes kh^{\mone} \cdot S(a_{i,g}))\\
&=\sum_{h' \in G}
h' \cdot a_{i,g}  \otimes e_{h'^{\mone}} + \sum_{h'' \in G}e_{h''} \otimes h''^{\mone} \cdot S(a_{i,g})\\
&= \sum a^{(1)} \otimes S(a^{(2)}), 
\end{split}
\end{equation*}
where we have changed the summation to $h'=kh$ and $h'' = hk^{\mone}$.
This proves our claim.
\end{proof}
 
 \begin{proposition}\label{Gamma-all-equiv}
 For all $M \in \cR ep_A(\bullet, \bullet)$, $\Gamma(M)$ carries a $G$-equivariant structure in $\cB_A(\bullet,\bullet)$. Moreover for any $f\colon M\to N$ in $\cR ep_A(\bullet,\bullet)$, the morphism $\Gamma(f)$ is $G$-equivariant.
 \end{proposition}
 \begin{proof}
Let $M \in \cR ep_A(\bullet, \bullet)$ and let
$$ \bigoplus_{i=1}^s Ae_{h_{i}} \xrightarrow{(\rho_{ij})} \bigoplus_{j=1}^t Ae_{g_{j}} \twoheadrightarrow M$$
be a projective presentation for $M$.
Since $\Gamma$ is exact,
$$ \Gamma(\bigoplus_{i=1}^s Ae_{h_{i}}) \xrightarrow{\Gamma(\rho_{ij})} \Gamma(\bigoplus_{j=1}^t Ae_{g_{j}}) \twoheadrightarrow \Gamma(M)$$
is a projective presentation for $\Gamma(M)$.
Defining $\alpha_k^{\oplus \Gamma(Ae_{g_{j}})}: \bigoplus_{j=1}^t \Gamma(Ae_{g_{j}}) \to \bigoplus_{j=1}^t \Gamma(Ae_{g_{j}})^k$ by applying $\alpha_k$ to each component, we obtain the commutative diagram
$$
\xymatrix{
\bigoplus_{i=1}^s \Gamma(Ae_{h_{i}})  \ar^{(\Gamma(\rho_{ij}))}[r]\ar^{\alpha_k^{\oplus\Gamma(Ae_{h_{i}})}}[d]& \bigoplus_{j=1}^t \Gamma(Ae_{g_{j}}) \ar^{\alpha_k^{\oplus\Gamma(Ae_{g_{j}})}}[d] \ar@{->>}[r] & \Gamma(M) \ar@{.>}[d]\\
\bigoplus_{i=1}^s \Gamma(Ae_{h_{i}}) ^k \ar^{\Gamma(\rho_{ij})^k}[r] & \bigoplus_{j=1}^t \Gamma(Ae_{g_{j}})^k \ar@{->>}[r] &\Gamma(M)^k,
}$$

where the solid vertical arrows are isomorphisms, implying that the induced morphism on the cokernels, which we define to be $\alpha_k^M$, is also an isomorphism as needed. The fact that these isomorphisms $\alpha_k^M$, for $k\in G$, make the diagram \eqref{G-equivobj} commute, follows from the same fact for the $\Gamma(Ae_g)$. Moreover, for any $f\colon M\to N$ in $\cR ep_A(\bullet,\bullet)$, one checks that $\Gamma(f)$ is $G$-equivariant by lifting $f$ to a projective presentation and applying Lemma \ref{Gamma-proj-equiv}.
 \end{proof}

\section{Symmetric bimodules and their simple transitive birepresentations}\label{s3}

\subsection{Symmetric bimodules}\label{s3.1}
Let now again $A$ be any finite-dimensional algebra and assume $G$ is a finite subgroup of the automorphism group of $A$ as in Section \ref{groupactions}.

Recall the category $\cB_A(\bullet,\bullet)$ of $A$-$A$-bimodules. 
Let $\mathcal{X}_A$ be the category whose objects are those of  $\cB_A(\bullet,\bullet)$, but in which morphism spaces between
objects $M$ and $N$ are given by
\begin{displaymath}
\mathrm{Hom}_{\mathcal{X}_A}(M,N):=
\bigoplus_{g\in G}\mathrm{Hom}_{A\text{-}A}(M,N^{g}).
\end{displaymath}
Thus, any $\varphi \in \mathrm{Hom}_{\mathcal{X}_A}(M,N)$ is given by a tuple
$(\varphi_{g})_{g\in G}$ such that $\varphi_{g}\in\mathrm{Hom}_{A\text{-}A}(M,N^{g})$.
Composition of $\varphi\in \mathrm{Hom}_{\mathcal{X}_A}(M,N)$ and $\psi\in  \mathrm{Hom}_{\mathcal{X}_A}(N,K)$
is defined by
\begin{equation}\label{eq1}
\begin{array}{ccc}
\mathrm{Hom}_{\mathcal{X}_A}(N,K)\otimes\mathrm{Hom}_{\mathcal{X}_A}(M,N)
&\to& \mathrm{Hom}_{\mathcal{X}_A}(M,K)\\
(\psi_{h})_{h\in G}\otimes (\varphi_{g})_{g\in G}&\mapsto& \big(\displaystyle\sum_{g\in G}(\psi_{sg^{\mone}})^{g}\circ \varphi_{g}\big)_{s\in G}.
\end{array}
\end{equation}
See \cite{CM}, where this is defined in Definition 2.3 and called a skew-catgory, for more details.

Denote by $\tilde{\mathcal{X}_A}$ the idempotent completion of $\mathcal{X}_A$, that is, objects of $\tilde{{\mathcal{X}_A}}$ are pairs $(M,e)$ where $M\in \X_A$ and $e=e^2\in\mathrm{End}_{\mathcal{X}_A}(M)$. For any $A$-$A$-bimodule $M$, we denote its associated stabiliser subgroup by
\begin{displaymath}
G_M:=\{g\in G\,\vert\, M\cong M^{g}\}.
\end{displaymath}


The following lemma is proved in exactly the same way as in the case of abelian $G$, see \cite[Lemma 2(i)]{MMZ2}.

\begin{lemma}\label{lem0}
{\hspace{1mm}}
For indecomposable $M\in \X_A$, there is an isomorphism of algebras
\begin{displaymath}
\mathrm{End}_{\mathcal{X}_A}(M)/\mathrm{Rad}(\mathrm{End}_{\mathcal{X}_A}(M))\cong
\Bbbk[{G}_M]/\mathrm{Rad}(\Bbbk[G_M])\cong\Bbbk[{G}_M].
\end{displaymath}
\end{lemma}
%

%

For any $M$, the group algebra $\Bbbk[G_M]$ is semi-simple and admits a unique decomposition
into a product of matrix rings. Let $\Bbbk[G_M]=S_1^M\oplus\cdots \oplus S_{s_M}^M$ be a decomposition into a direct sum of simple modules, and let $\{\tilde\eps_1^M,\dots \tilde\eps_{s_M}^M\}$ be the corresponding set of
primitive idempotents in $\Bbbk[G_M]$. Assume that $S_1^M$ is the trivial module. Each $\tilde\eps^M_i$ has the form
{\small$\displaystyle\frac{1}{|G_M|}\sum_{g\in G_M}\lambda^M_i(g)g$} for some scalars $\lambda^M_i(g)$ and hence defines an idempotent
${\eps}_i$ in $\mathrm{End}_{\mathcal{X}_A}(M)$ given by
the tuple {\small $\left(\frac{\lambda_i^M(g)}{|G_M|}g\right)_{g\in G_M}$}.
In the special case of $M=A$, we omit the sub- and superscripts $M$, set $s=s_A$ and also write $\tilde\pi_i=\tilde\eps_i^A$ and $\pi_i=\eps_i^A$.

It immediately follows from the definitions that the indecomposable objects of $\tilde{{\mathcal{X}_A}}$ are of the form
$(M,\varepsilon_{j}^M)$, where
$M$ is indecomposable as an $A$-$A$-bimodule and $j=1, \dots, s_M$. Moreover, 
$(M,\varepsilon_{i}^M)$ and $(M,\varepsilon_{j}^M)$ are isomorphic if and only if $S_i^M\cong S_j^M$.


%
%
%

In order to arrive at a bicategory whose $1$-morphisms are the objects of $\tilde \X_A$, we equip $\mathcal{X}_A$ with a tensor product by setting
\begin{itemize}
\item $\displaystyle M\otimes^G N=\bigoplus_{g\in G}
\big(M^{g^{}}\otimes_A N\big)$, for any $M, N\in \X_A$, and
\item $\varphi\otimes^G \psi=\big( (\varphi_{g k^{\mone}})^k\otimes \psi_{h}\big)_{g,h,k\in G}$,
where
\begin{displaymath}
(\varphi_{g k^{\mone}})^{k}\otimes \psi_{h}:
M^{k^{}}\otimes_A N\to (M')^{g}
\otimes_A (N')^{h},
\end{displaymath}
for $M, M',N,N' \in \X_A$ and 
$
\varphi\in\mathrm{Hom}_{\mathcal{X}_A}(M,M'),
\psi\in\mathrm{Hom}_{\mathcal{X}_A}(N,N').
$
\end{itemize}

Since $(M'^{g h^{\mone}}\otimes_A N')^{h}\cong (M')^{g}\otimes_A (N')^{h}$, we observe that $(\varphi_{g k^{\mone}})^{k}\otimes \psi_{h}$ is a component of $(\varphi\otimes^G \psi)_{h}$.
The following lemma is the analogue of \cite[Lemma 3]{MMZ2} and shows that the asymmetry in the defintion of $-\otimes^G-$ is only notational. For the reader's convenience, we include the proof.

\begin{lemma}\label{lem3}
There is an isomorphism
\begin{displaymath}
\bigoplus_{g\in G}
\big(M^{g^{}}\otimes_A N\big) \cong
\bigoplus_{g\in G}
\big(M\otimes_A N^{g^{\mone}}\big)
\end{displaymath}
in $\X_A$.
\end{lemma}

\begin{proof}
By, Lemma \ref{canoniso}, there is a canonical isomorphism
\begin{equation}\label{eq2}
M\otimes_A N^{g^{\mone}}\cong
\left(M^{g}\otimes_AN\right)^{g^{\mone}}
\end{equation}
in $\cB_A(\bullet, \bullet)$. Thus, 
\begin{displaymath}
\xymatrix{M\otimes_A N^{g^{\mone}}\ar[rr]^{(\varphi_{h})_{h\in G}}&&
M^{g}\otimes_A N},
\end{displaymath}
with $\varphi_{g^{\mone}}$ given by \eqref{eq2} and the remaining components chosen as zero, defines the required isomorphism in $\X_A$.
\end{proof}


The following lemma is the non-abelian version of \cite[Lemma 5]{MMZ2} and proved analogously.

\begin{lemma}\label{lem2}
{\hspace{1mm}}

\begin{enumerate}[$($i$)$]
\item\label{lem2.1}
The operation $- \otimes^G -$ is bifunctorial.
\item\label{lem2.2}
If $e$ and $f$ are idempotents in $\mathcal{X}_A$, then so is $e\otimes^G f$.
Hence $- \otimes^G -$ extends to a bifunctor
$- \otimes^G -:\tilde{\mathcal{X}_A}\times \tilde{\mathcal{X}_A}\to \tilde{\mathcal{X}_A}$
given by $(M,e)\otimes^G(N,f)=(M\otimes^GN,e\otimes^G f)$.
\end{enumerate}
\end{lemma}

We next obtain results akin to \cite[Propositions 6 and 7]{MMZ2}. We only consider $(A,{\pi}_{1})$ on the one hand, but generalise to any bimodule $M$ with any idempotent on the other hand. The proofs are similar to those in loc. cit. but since our setup, notation and level of generality are different, we include them for the reader's convenience.

\begin{proposition}\label{propmultid2}
$M\in \X_A$ and $l \in \{1,\dots, s_M\}$. Then
\begin{equation}\label{eq8}
(M,\varepsilon_{l}^M)\otimes^G(A,{\pi}_{1})\cong (M,\varepsilon_{l}^M) \cong (A,{\pi}_{1})\otimes^G(M,\varepsilon_{l}^M).
\end{equation}
\end{proposition}

\begin{proof}
We start by constructing a morphism $\varphi$ from the right hand side of \eqref{eq8} to the left hand side.
Consider the morphism
\begin{displaymath}
\varphi :=(\varphi_{s,t})_{s,t\in G}\colon M\to
\displaystyle\bigoplus_{s,t\in G}M^{s}\otimes_A A^{t}
\end{displaymath}
where $\varphi_{s,t}$ is given by
\begin{equation*}
m\mapsto \frac{1}{|G_M||G|}\lambda_l^M(s)(s(m)\otimes 1)\in
M^{s}\otimes_A A^{t}.
\end{equation*}
if $s\in G_M$ and zero otherwise.

Consider the diagram
\begin{displaymath}
\xymatrix{
M\ar[rrr]^{\varepsilon_{l}^M}\ar[d]_{(\varphi_{s,t})_{s,t\in G}}
&&&\displaystyle\bigoplus_{h\in G}M^{h}
\ar[d]^{((\varphi_{gh^{\mone},kh^{\mone}})_{gh^{\mone},kh^{\mone}\in G}^{h})_{h\in G}}\\
\displaystyle\bigoplus_{s,t\in G}M^{s}\otimes_A A^{t}
\ar[rrr]^{(({\varepsilon}_{l}^M\otimes^G{\pi}_{1})^t)_{t\in G}}
&&&\displaystyle
\bigoplus_{g,k\in G}
M^g \otimes_A A^k.
}
\end{displaymath}
By definition,  the $st^{\mone},gt^{\mone},kt^{\mone}$-component of
$\varepsilon_{l}^M\otimes^G{\pi}_{1}$
sends $m\otimes 1$ to
$$
\frac{1}{|G_M||G|}\lambda_l^M(gt^{\mone})
(gt^{\mone}(m)\otimes 1)
$$
if $gt^{\mone}\in G_M$ and zero otherwise.
Now, going to the right and then down, the $g,k$-component of the composition
$\varphi\circ \varepsilon_{l}^M$
sends $m\in M$ to
\begin{displaymath}
\begin{split}
m \mapsto \sum_{h \in G_M} \frac{1}{|G_M|} \lambda^M_l(h) h(m) \mapsto 
 \sum_{h \in G_M} \frac{1}{|G_M|} \lambda^M_l (h) \frac{1}{|G_M|} \frac{1}{|G|} \lambda^M_l(gh^{\mone})gh^{\mone}(h(m) \otimes 1) 
 & \\
= \sum_{h \in G_M} \frac{1}{|G||G_M|^2} \lambda^M_l(h) \lambda^M_l(gh^{\mone}) (g(m) \otimes 1) = \frac{1}{|G_M||G|}\lambda_l^M(g)(g(m)\otimes 1),
\end{split}
\end{displaymath}

if $g$ is in $G_M$, and zero otherwise. Going down and then to the right,
the $g,k$-component of
$(\varepsilon_{l}^M\otimes^G{\pi}_{1})\circ \varphi$ yields
\begin{displaymath}
\begin{split}
m \mapsto \sum_{s\in G_M} \frac{1}{|G||G_M|} \lambda^M_l(s) (s(m) \otimes 1) \mapsto \sum_{s \in G_M, t \in G}\frac{1}{|G|^2 |G_M|^2} \lambda^M_l(s) \lambda^M_l(gs^{\mone})gs^{\mone}(s(m)\otimes 1)
& \\
= \sum_{t \in G} \frac{1}{|G|^2 |G_M|} \lambda^M_l(g)(g(m) \otimes 1) = \frac{1}{|G||G_M|} \lambda^M_l(g) (g(m) \otimes 1).
\end{split}
\end{displaymath}
Hence, the diagram commutes.

To construct a morphism $\psi$ from the left hand side of \eqref{eq8} to the right hand side,
consider the diagram
\begin{displaymath}
\xymatrix{\displaystyle\bigoplus_{g\in G}(M^{g} \otimes_A A)
\ar[rrr]^{\varepsilon^M_{l}\otimes^G{\pi}_{1}}
\ar[d]_{(\psi_{h,g})_{g,h \in G}}
&&&
\displaystyle\bigoplus_{u,k\in G} M^{u}\otimes_A A^{k}
\ar[d]^{\big( (\psi_{sk^{\mone},uk^{\mone}})_{sk^{\mone},uk^{\mone}\in G}^{k}\big)_{k\in G}}
\\
\displaystyle\bigoplus_{h\in G}  M^{h}
\ar[rrr]^{( \varepsilon^M_{l})^{h}_{h\in G}}
&&&
\displaystyle\bigoplus_{s\in G}M^{s}.
}
\end{displaymath}
where $\psi_{h,g}$ sends $m\otimes 1\in M^{g}\otimes_A A$
to $\frac{1}{|G_M|}\lambda_l^M(hg^{\mone})h g^{\mone}(m)\in M^{h}$,
if $hg^{\mone}\in G_M$, and to zero otherwise.

Fix $g,s\in G$. Going down and then to the right we obtain
\begin{displaymath}
\begin{split}
 m \otimes 1 \mapsto \sum_{h \in G_M} \frac{1}{|G_M|} \lambda^M_l(hg^{\mone})hg^{\mone}(m)  \mapsto \sum_{h \in G_M} \frac{1}{|G_M|^2} \lambda^M_l(hg^{\mone}) & \lambda_l^M(sh^{\mone}) sh^{\mone}(hg^{\mone}(m)) 
  \\
 =\frac{1}{|G_M|} \lambda_l^M(sg^{\mone})sg^{\mone}(m),
\end{split}
\end{displaymath}
where the last equality follows from the fact that $\varepsilon_l$ is an idempotent.

Next we calculate the composition first going right and then down:

\begin{displaymath}
\begin{split}
 & m \otimes 1  \mapsto \sum_{u \in G_M, k \in G} \frac{1}{|G_M||G|} \lambda_l^M(ug^{\mone})(ug^{\mone}(m) \otimes 1) 
  \\
& \mapsto \sum_{u\in G_M, k \in G} \frac{1}{|G_M|^2} \frac{1}{|G|}  \lambda_l^M(ug^{\mone})  \lambda^M_l(su^{\mone})  su^{\mone}(ug^{\mone}(m)) 
\\
& =\sum_{k\in G} \frac{1}{|G_M|} \frac{1}{|G|}\lambda_l^M(sg^{\mone})sg^{\mone}(m)=\frac{1}{|G_M|} \lambda_l^M(sg^{\mone})sg^{\mone}(m).
\end{split}
\end{displaymath}

Thus, the diagram commutes and $\psi$ is well-defined.


Now we claim that both compositions $\varphi\circ\psi$ and $\psi \circ \varphi$
are the identities, i.e. the respective idempotents. The $k$-component of the
composition $\psi \circ \varphi$ sends $m$ to

\begin{displaymath}
\begin{split}
\sum_{s \in G_M, t \in G} \frac{1}{|G_M|^2 |G|} \lambda^M_l(s) \lambda^M_l (ks^{\mone}) ks^{\mone}(s(m)) 
 = \frac{1}{|G_M|} \lambda^M_l(k)k(m).
\end{split}
\end{displaymath}

The $g,s,t$-component of the composition $\varphi \circ \psi$ sends $m\otimes 1 \in M^g \otimes_A A$  to

\begin{displaymath}
\begin{split}
 \sum_{sk^{\mone}, kg^{\mone} \in G_M} \frac{1}{|G_M|^2 |G|}\lambda^M_l(kg^{\mone})\lambda^M_l(sk^{\mone}) sk^{\mone}(kg^{\mone} \otimes 1) \\
 = \frac{1}{|G_M||G|} \lambda^M_l (sg^{\mone}) (sg^{\mone} \otimes 1)
\end{split}
\end{displaymath}
in $M^s \otimes_A A^t$.
The first isomorphism in \eqref{eq8} follows and the second is proved analogously.

\end{proof}

\subsection{The bicategories $\cX_A$ and $\cG_{A}$}

We use the data above to define a bicategory $\cX_A$ with
\begin{itemize}
\item one object $\bullet$;
\item $\cX_A(\bullet,\bullet) = \tilde \X_A$;
\item horizontal composition given by the tensor product $- \otimes^{G} - $ in $\tilde{\X}_A$;
\item the identity $1$-morphism given by $(A,{\pi}_{1})$;
\item for each triple of $1$-morphisms $X,Y,Z$ an associatior $\alpha_{X,Y,Z}: (X \otimes^{G} Y) \otimes^{G} Z \to X \otimes^{G} (Y \otimes^{G} Z)$, induced by the standard associator in the category of $A$-$A$-bimodules (see Section \ref{bimodconv});
\item left and right unitors $\lambda_X: (A, {\pi_1})  \otimes^{G} X \xrightarrow{\cong} X$ and $\rho_X: X \otimes^{G} (A, {\pi_1}) \xrightarrow{\cong} X$ induced by the isomorphisms in Proposition 
 \ref{propmultid2}.
\end{itemize}

Since $(A,{\pi}_{1}) \oplus(A\otimes_{\Bbbk}A)$ is, up to isomorphism, invariant under the twist
$M\mapsto M^{g}$ for $g\in G$, and its additive closure is closed under horizontal composition, we can define a bicategory $\cG_{A}$ as the $2$-full subbicategory of $\cX_A$ whose $1$-morphisms are given by objects  in the additive closure of
$(A,{\pi}_{1}) \oplus(A\otimes_{\Bbbk}A)$ inside $\tilde{\X}_A$. This is finitary by construction.

%

\subsection{Embedding $\cX_A$ into $\cB_A$}

Define a functor $\Theta\colon \mathcal{X}_A\to \cB_A(\bullet,\bullet)$ by $M \mapsto \bigoplus_{g \in G} M^g$ on objects and by mapping $f: M \to N$ to $\Theta(f)= (f^g_{hg^{\mone}}): \bigoplus_{g \in G} M^g \to \bigoplus_{h \in G} N^h$ on morphisms. Notice that $\Theta$ is faithful by construction.

Using the definition on morphisms we can extend $\Theta$ to the category $\tilde{\X}_A$. Recall that objects of $\tilde{\X}_A$ are pairs $(M,e)$ where $M$ is an $A$-$A$-bimodule and $e$ is an idempotent in $\End_{\X_A}(M)$. Thus, we extend $\Theta: \tilde{\mathcal{X}_A}\to \cB_A(\bullet,\bullet)$ via $(M,e) \mapsto \Theta(e)\Theta(M)$.

\begin{lemma}\label{thetapseudo}
$\Theta$ defines a pseudofunctor from $\cX_A$ to $\cB_A$.
\end{lemma}

\proof

Note that for $M,N\in \mathcal{X}_A$
$$\Theta(M \otimes^G N)= \Theta \big(\bigoplus_{g \in G} M^g \otimes_A N \big) = \bigoplus_{h \in G} \bigoplus_{g \in G} (M^g \otimes_A N)^h = \bigoplus_{g,h \in G} M^{gh} \otimes_A N^h = \bigoplus_{k,h \in G} M^k \otimes_A N^h,$$

where the penultimate equality is true by Lemma \ref{canoniso}. 

On the other hand, $$\Theta(M) \otimes_A \Theta(N) = \bigoplus_{g \in G} M^g \otimes_A \bigoplus_{h \in G} N^h.$$

Thus, we have a natural isomorphism,
$$ J_{M,N}: \Theta(M) \otimes_A \Theta(N) \to \Theta(M \otimes^G N),$$
compatible with associativity and unitors, which, by functoriality, extends to $\tilde \X_A$.
\endproof

\begin{lemma}\label{thetaimage}
The essential image of $\Theta$ consists of all $G$-equivariant bimodules.
\end{lemma}

\proof
Let $(N, e) \in \tilde{\X}_A$ and set $M=\Theta(N)=\bigoplus_{h \in G} N^h$. Thus, we have a well-defined  isomorphim $\alpha_g: \bigoplus_{h \in G} N^h \to \bigoplus_{h \in G} N^{gh} $, which is given by relabelling the components. Since this is essentially the identity map, it is clear that this produces a $G$-equivariant structure on $M$. This commutes withe the idempotent $\Theta(e)$ and hence also defines a $G$ equivariant structure on $\Theta(N,e)$.

Conversely,  suppose that $M$ is a $G$-equivariant object of $\cB_A(\bullet,\bullet)$. Let $\alpha_g: M \to M^g$ be the corresponding isomorphism defining the equivariant structure of $M$.
Viewing $M$ as an object of $\X_A$, consider $\Theta(M)=\bigoplus_{g \in G} M^g$ and let $e \in \End_{\X_A} (M)$ be defined by $e=\frac{1}{|G|} (\alpha_g)_{g \in G}$. Notice that 
 $$ e \circ e = \frac{1}{|G|^2} \sum_s \alpha^s_{hs^{\mone}} \alpha_s = \frac{1}{|G|} \alpha_h=e.$$
Thus, $e$ is an idempotent. 

Consider the map $\bar{\alpha}: M \to \Theta(M)$ given by $m \mapsto (\alpha_g(m))_{g \in G}$. Since $e$ is an idempotent in $\End_{\X_A}(M)$, $\Theta(e)$ is an idempotent in $\Theta(M)$ so we have an embedding $\Theta(e)\Theta(M) \to \Theta(M)$. We claim that $\bar{\alpha}$ factors over $\Theta(e) \Theta(M)$, i.e. $ (\alpha_g(m))_{g \in G} \in \Theta(e) \Theta(M)$. Indeed,

\begin{equation*}
\begin{split}
\Theta(e)(\alpha_g(m))_{g \in G}=& \frac{1}{|G|}(\alpha_{st^{\mone}})_{s,t \in G}(\alpha_g(m))_{g \in G}  \\
=&  \frac{1}{|G|} \sum_{g \in G} (\alpha^g_{sg^{\mone}}\alpha_g(m))_{s \in G}=(\alpha_s(m))_{s \in G}.
\end{split}
\end{equation*}

Therefore, $\Theta(e)(\alpha_g(m))_{g \in G}=(\alpha_g(m))_{g \in G}$ and hence $(\alpha_g(m))_{g \in G} \in \Theta(e)\Theta(M)$ as required. 

Moreover, the map $\bar{\alpha}$ has an inverse $\bar{\beta}: \Theta(M) \to M$ given by $(\alpha_g(m))_{g \in G} \mapsto m$. 
This shows that $\Theta(e)\Theta(M)=\Theta(M,e) \cong M$, completing the proof.
\endproof

\begin{corollary} \label{thetares}
$\Theta$ restricts to a pseudofunctor $\cG_A\to\cC_A$, whose essential image consists of all $G$-equivariant $1$-morphisms.
\end{corollary}

\proof 
This follows from $\Theta(A, {\pi}_1) \cong A$ by the proof of Lemma \ref{thetaimage}, combined with the fact that $\Theta$ maps projective objects to projective objects.
\endproof

\begin{theorem}\label{HAinGA}
Let $A$ be a finite-dimensional radically graded basic Hopf algebra and adopt the conventions from Section \ref{radgrbasichopf}. 
Identifying $\cR ep_A$ and $\cX_A$ with their essential images in $\cB_A$ under the pseudofunctors $\Gamma$ and $\Theta$, respectively, $\cR ep_A$ can be viewed a subbicategory of $\cX_A$.

Under the same identification, $\cH_A$ corresponds to a $1$-full subbicategory of $\cG_A$.
\end{theorem}

\proof
The first statement follows from Proposition \ref{Gamma-all-equiv}, Lemma \ref{thetapseudo} and Lemma \ref{thetaimage}.

For the second statement, note that $A$ being a finite-dimensional radically graded basic Hopf algebra implies that the action of $G$ on the set of idempotents is regular. Thus, the $1$-morphism $Ae_g\otimes e_h A$, for $g,h\in G$, in $\cG_A$ are indecomposable (and isomorphic to $Ae_{gh^{-1}}\otimes e_1 A$). Now  the indecomposable $1$-morphisms in the essential images of both $\cH_A$  and $\cG_A$ are those isomorphic to the identity or to $\bigoplus_{h\in G} Ae_{gh}\otimes e_h A$ for some $g\in G$.
\endproof

\subsection{The bicategory $\widetilde{\cG_A}$}\label{s3.3}

Assume that we are in the setup of Subsection~\ref{s3.1}. Further, we assume that $A$ is basic and has a fixed complete $G$-invariant set $\mathtt{E}$ of primitive idempotents such that $G$ acts regularly on $\mathtt{E}$. We can thus choose an idempotent $e_1$ as a base point and label all other idempotents by group elements, obtaining $\mathtt{E} = \{e_g | g \in G\}$.
This labelling is chosen such that $h(e_g) = e_{gh^{\mone}}$. 

Write $A=\Bbbk Q/I$ for a quiver $Q$ and admissible ideal $I$. Let $A_0=\Bbbk Q_0$ and define $\hat A= A\times A_0$. To distinguish the idempotents in $\hat A$ coming from the copies of $A$ and $A_0$ respectively, given an idempotent $e_g$ in $A$, we denote the corresponding idempotent in $A_0$ by $\bar e_g$.

The group action of $G$ on $A$ extends naturally to an action on $\hat A$ by leaving the two factors invariant and permuting the idempotents in $A_0$, so we can consider the categories $\X_{\hat A}$ and  $\tilde \X_{\hat A}$ of symmetric ${\hat A}$-${\hat A}$-bimodules. Note that $\X_A$ and  $\tilde \X_A$ (and similarly $\X_{A_0}$ and  $\tilde \X_{A_0}$) can be viewed as full subcategories of $\X_{\hat A}$ and  $\tilde \X_{\hat A}$, respectively.

\begin{lemma}\label{projBaction}
For all $g,g',h\in G$, we have isomorphisms $$(A e_g\otimes_{\Bbbk}e_{g'}A)^h\cong Ae_{gh}\otimes_{\Bbbk}e_{g'h}A  \qquad (A_0\bar e_g\otimes_{\Bbbk}\bar e_{g'}A_0)^h\cong A_0\bar e_{gh}\otimes_{\Bbbk}\bar e_{g'h}A_0$$
$$(Ae_g\otimes_{\Bbbk}\bar e_{g'}A_0)^h\cong Ae_{gh}\otimes_{\Bbbk}\bar e_{g'h}A_0 \qquad(A_0\bar e_g\otimes_{\Bbbk}e_{g'}A)^h\cong A_0\bar e_{gh}\otimes_{\Bbbk}e_{g'h}A .$$
Moreover, for $g,g'\in G$, the ${\hat A}$-${\hat A}$-bimodules $A e_g\otimes_{\Bbbk}e_{g'}A$, $A_0\bar e_g\otimes_{\Bbbk}\bar e_{g'}A_0$,  $Ae_g\otimes_{\Bbbk}\bar e_{g'}A_0$ and $A_0\bar e_g\otimes_{\Bbbk}e_{g'}A$ are indecomposable in $\tilde{X}_{\hat A}$.
\end{lemma}

\proof
The description of the $G$-action on the first two modules is given by Lemma \ref{actiononproj} and the natural inclusions of  $\tilde \X_A$ and  $\tilde \X_{A_0}$ into  $\tilde \X_{\hat A}$. The third and fourth isomorphism are verified by analogous computations. Given freeness of the action of $G$ on the chosen set of primitive orthogonal idempotents of ${\hat A}$, the statement about indecomposability follows from Lemma \ref{lem0}.
\endproof

\begin{lemma}\label{A0id}
Let $M$ be a projective right $A\times A_0$-module, and $N$ a projective left $A\times A_0$-module. Then there are canonical isomorphisms
$$(\Bbbk  \bar e_1\otimes_\Bbbk  \bar e_1\Bbbk )\otimes^{G} (\Bbbk  \bar e_g \otimes_\Bbbk M) \cong \Bbbk  \bar e_g \otimes_\Bbbk M$$
and 
$$ (N\otimes_\Bbbk  \bar e_g\Bbbk)\otimes^{G} (\Bbbk  \bar e_1\otimes_\Bbbk  \bar e_1\Bbbk )\cong (N\otimes_\Bbbk  \bar e_g\Bbbk).$$
\end{lemma}

\proof
We have canonical isomorphisms
\begin{equation*}\begin{split}
(\Bbbk  \bar e_1\otimes_\Bbbk  \bar e_1\Bbbk )\otimes^{G} (\Bbbk  \bar e_g \otimes_\Bbbk M) &= \bigoplus_{h\in G} (\Bbbk  \bar e_{h}\otimes_\Bbbk  \bar e_{h}\Bbbk )\otimes_{A_0} (\Bbbk  \bar e_g \otimes_\Bbbk M)\\
&\cong \Bbbk  \bar e_g \otimes_\Bbbk  \bar e_g\Bbbk \otimes_{A_0} \Bbbk  \bar e_g \otimes_\Bbbk M\\
&\cong \Bbbk  \bar e_g \otimes_\Bbbk M
\end{split}\end{equation*}
and the second isomorphism in the lemma follows similarly.
\endproof

 We now define a bicategory $\widetilde{\cG_A}$ as having
\begin{itemize}
\item two objects $\bullet$ and $\ast$;
\item $\widetilde{\cG_A}(\bullet,\bullet)$ is simply $\cG_A(\bullet,\bullet)$; 
\item $\widetilde{\cG_A}(\bullet,\ast)$ is given by bimodules in the additive closure of $A_0\otimes_\Bbbk A$ inside $\tilde\X_{\hat A}$;
\item $\widetilde{\cG_A}(\ast,\bullet)$ is given by bimodules  in the additive closure of $A\otimes_\Bbbk A_0$ inside $\tilde\X_{\hat A}$;
\item $\widetilde{\cG_A}(\ast,\ast)$ is given by bimodules  in the additive closure of $A_0\otimes_\Bbbk A_0 $ inside $\tilde\X_{\hat A}$;
\item the identity $1$-morphism on $\ast$ and the corresponding unitors are given by Lemma \ref{A0id};
\item $2$-morphisms and their composition are inherited from $\tilde\X_{\hat A}$;
\item horizontal composition and associators are inherited from $\tilde \X_{\hat A}$.
\end{itemize}

%
Set 
\begin{equation*}
\begin{split}
\S_{1}&:= \{(A,{\pi}_{1})\}\\
\S_{11}&:=\{Ae_g\otimes_{\Bbbk}e_1A \vert g\in G\}\\
\S_{01}&:=\{A_0 \bar e_g\otimes_{\Bbbk}e_1A \vert g\in G\}\\
\S_{10}&:=\{Ae_g \otimes_{\Bbbk} \bar e_1 A_0 \vert g\in G\}\\
\S_{00}&:= \{A_0 \bar e_g\otimes_{\Bbbk} \bar e_1A_0\vert g\in G\}.
\end{split}
\end{equation*}

\begin{lemma} \label{ind1morph}
Up to isomorphism, the indecomposable $1$-morphisms in $\widetilde{\cG_A}$ are given by bimodules in 
$$\S_{1} \cup\S_{11} \cup\S_{01} \cup \S_{10}\cup\S_{00}.$$
\end{lemma}

\proof
First note that by Lemma \ref{projBaction} for the objects in $\S_{11} \cup\S_{01} \cup \S_{10}\cup\S_{00}$ and by construction for $(A,{\pi}_{1})$, the given bimodules are indeed indecomposable $1$-morphisms in $\widetilde{\cG_A}$.

Also by construction and Lemma \ref{A0id}, each indecomposable $1$-morphism in $\widetilde{\cG_A}$ is isomorphic to $(A,{\pi}_{1})$ or one of $A \hat e_h\otimes_{\Bbbk}\hat e_{h'}A$ for $h,h'\in G$, with $\hat e_h\in\{e_h, \bar e_h\}$ for $h\in G$. It thus suffices to show that each of the $A \hat e_h\otimes_{\Bbbk}\hat e_{h'}A$ is isomorphic to one of the bimodules in $\S_{11} \cup\S_{01} \cup \S_{10}\cup\S_{00}$ in $\tilde{\X}_{\hat A}$. This, however, is a consequence of Lemma \ref{projBaction}, since $(A \hat e_h\otimes_{\Bbbk}\hat e_{h'}A)^{h'^{-1}}\cong A \hat e_g\otimes_{\Bbbk}\hat e_{1}A$ as ${\hat A}$-${\hat A}$-bimodules.
\endproof

\subsection{Two-sided cells in $\cG_{A}$ and $\widetilde{\cG_A}$}\label{s3.4}

We keep the assumptions from Section \ref{s3.3} and recall the notation introduced after Lemma~\ref{lem0}.

\begin{proposition}\label{prop7}
The bicategory $\cG_{A}$ has two $\H$-cells, which are also two-sided cells, namely

\begin{enumerate}[$($a$)$]
\item\label{prop7.1} $\mathcal{H}_1$
consisting of one element $(\mathbbm{1}_{\bullet})$;
\item\label{prop7.2}
the two-sided cell $\mathcal{H}_{0}$ consisting of all isomorphism classes of
$1$-morphisms in $\S_{11}$ .
\end{enumerate}
\end{proposition}

\begin{proof}
Since tensor products in which one of the factors is a  projective bimodule never contain a copy of the regular bimodule as a direct summand, it follows immediately that the indecomposable $1$-morphism in $\S_1$ is strictly smaller than those in $\S_{11}$ in the left right and two-sided order. This shows the existence of an $\H$-cell $\H_1$, as claimed in \eqref{prop7.2}, which is a left, right and two-sided cell.
To complete the proof, it suffices to show that the indecomposable $1$-morphisms in $\S_{11}$ are in the same left and the same right cell. Let $Ae_g\otimes_{\Bbbk}e_1A$ and $Ae_h\otimes_{\Bbbk}e_1A$ be two representatives of isomorphism classes of indecomposable $1$-morphisms in $\S_{11}$. Then $Ae_g\otimes_{\Bbbk}e_1A$ is a direct sumand of $(Ae_{gh^{\mone}}\otimes_{\Bbbk}e_1A)\otimes_{\tilde{\X}}(Ae_h\otimes_{\Bbbk}e_1A)$, hence $Ae_g\otimes_{\Bbbk}e_1A\leq_L Ae_h\otimes_{\Bbbk}e_1A$. Similarly, $Ae_h\otimes_{\Bbbk}e_1A\leq_L Ae_g\otimes_{\Bbbk}e_1A$ since $Ae_h\otimes_{\Bbbk}e_1A$ is a direct sumand of $(Ae_{hg^{\mone}}\otimes_{\Bbbk}e_1A)\otimes_{\tilde{\X}}(Ae_g\otimes_{\Bbbk}e_1A)$, so they are in the same left cell. Similarly, one sees that they are in the same right, and hence in the same $\H$-cell, which then is also a two-sided cell.
\end{proof}

Recall that $\cH_A$ can be viewed as a $1$-full subbicategory of $\cG_A$ and hence shares its $\H$-cell structure.

\begin{proposition}\label{Hsimple}
$\cG_A$ and $\cH_A$ are $\H_0$-simple.
\end{proposition}

\proof
Assume $\cI$ is a nonzero $2$-ideal in $\cG_A$, identified with a subbicategory of $\cC_A$ under $\Theta$, and let $f\colon M\to N$ be a morphism in $\cI(\bullet, \bullet)$. Note that $A\otimes_\Bbbk A$ is a $1$-morphism in the essential image of $\cG_A$ under $\Theta$.  Then $\id_{A\otimes_\Bbbk A} \otimes f \otimes \id_{A\otimes_\Bbbk A}$ contains an identity component on a $1$-morphism in $\cG_A$, as required.

The proof for $\cH_A$ is analogous.
\endproof

\begin{proposition}\label{prop7.5}
The bicategory $\widetilde{\cG_{A}}$ has two two-sided cells, namely

\begin{enumerate}[$($a$)$]
\item\label{prop7.51} $\mathcal{J}_1$
consisting of one element $(\mathbbm{1}_{\bullet})$;
\item\label{prop7.52}
the two-sided cell $\widetilde{\mathcal{J}}_{0}$ consisting of all isomorphism classes of indecomposable
$1$-morphisms in $\S_{11} \cup\S_{01}\cup
  \S_{10}\cup\S_{00}$.
Furthermore, the two-sided cell $\widetilde{\mathcal{J}}_{0}$ consists of two left cells $\L_0 = \S_{10}\cup\S_{00}$ and $\L_1=\S_{11} \cup\S_{01}$ and two right cells $\R_0 = \S_{01}\cup
  \S_{00}$ and $\R_1 =\S_{11} \cup\S_{10} $. Hence, the $\H$-cells are $\H_{ij}=\S_{ij}$ for $i,j\in\{0,1\}$.
\end{enumerate}
\end{proposition}

\proof
Part \eqref{prop7.51} follows as in Proposition \ref{prop7} \eqref{prop7.1}, since tensor products involving projective $A\times A_0$ bimodules never contain direct summands isomorphic to the regular $A$-$A$-bimodule.

To prove part \eqref{prop7.52}, let $A_i\in \{A,A_0\}$ for $i=1,\cdots, 3$, and let $\hat e_h\in \{e_h,\bar e_h\}$, for $h\in G$, as appropriate. Then $A_1\hat e_g\otimes_{\Bbbk}\hat e_1A_2$ is a direct summand of $(A_1\hat e_{gh^{\mone}}\otimes_{\Bbbk}\hat e_1A_3)\otimes_{\tilde{\Y}}(A_3\hat e_h\otimes_{\Bbbk}\hat e_1A_2)$, so $A_1\hat e_g\otimes_{\Bbbk}\hat e_1A_2\leq_L A_3\hat e_h\otimes_{\Bbbk}\hat e_1A_2$ for any $A_3$. This shows that all $1$-morphisms in $\S_{10}\cup\S_{00}$ are in the same left cell and all $1$-morphisms in $\S_{11} \cup\S_{01}$ are in the same left cell. Given that tensoring $A_1\hat e_g\otimes_{\Bbbk}\hat e_1A_2$ on the left does not change $A_2$, it is clear that $\L_0 = \S_{10}\cup\S_{00}$ and $\L_1=\S_{11} \cup\S_{01}$ are two different left cells. Similarly, ones proves the statement about right cells. The statement about $\H$-cells and the two-sided cell $\widetilde{\J}_0$ follows immediately.
\endproof

\begin{lemma}\label{vecGequiv}
The endomorphism bicategory $\widetilde{\cG_{A}}_\ast$ of the object $\ast$ in $\widetilde{\cG_{A}}$ has a unique left, right and two-sided cell $\H_{00}$ and is biequivalent to $\cV ec_G$.
\end{lemma}

\proof
Given that the indecomposable $1$-morphisms in $\widetilde{\cG_{A}}_\ast$ are precisely those in $\H_{00}$, the first statement follows from Lemma \ref{prop7.5}. Now 
\begin{equation*}\begin{split}
(A_0\bar e_{g}\otimes_{\Bbbk}\bar e_1A_0)\otimes^{G}(A_0\bar e_h\otimes_{\Bbbk}\bar e_1A_0)&=\bigoplus_{k\in G}(A_0\bar e_{gk}\otimes_{\Bbbk}\bar e_kA_0)\otimes_A(A_0\bar e_h\otimes_{\Bbbk}\bar e_1A_0)
\\ & \cong A_0\bar e_{gh}\otimes_{\Bbbk}\bar e_1A_0.
\end{split}\end{equation*}
Therefore $\widetilde{\cG_{A}}_\ast$ is a semi-simple bicategory with one object, which decategorifies to $\mathbb{Z}G$. Moreover, the associator is induced by the associator on bimodules given in  Section \ref{bimodconv} which corresponds to the trivial $3$-cocycle and thus $\widetilde{\cG_{A}}_\ast$ is biequivalent to $\cV ec_G$, c.f. \cite[Proposition 4.10.3]{EGNO}. 
\endproof

\subsection{Adjunctions}\label{s3.6}

We keep the setup of Section \ref{s3.3} and additionally assume that $A$ 
is self-injective.
We denote by $\nu$ the bijection on $\mathtt{E}$
which is induced by the Nakayama automorphism of $A$ given by
\begin{displaymath}
\mathrm{Hom}_{\Bbbk}(eA,\Bbbk)\cong A\nu(e), \quad\text{ for } e\in  \mathtt{E}.
\end{displaymath}
Note that given our labelling of idempotents by elements of $G$, this induces and action of $\nu$ on $G$ and we define $\nu(g)$ by $\nu(e_g)=e_{\nu(g)}$.

\begin{lemma}\label{nucompat}
For any $g\in G$, we have $\nu(g)=\nu(1)g$. 
\end{lemma}

\proof
By definition, $Ae_{\nu(1)g} \cong (Ae_{\nu(1)})^{g}\cong \Hom_\Bbbk(e_1A, \Bbbk)^{g} $ is the indecomposable injective $A$-module whose socle element $e_1^*$ satisfies
$e_h\cdot (e_1)^*  =g(e_h)(e_1)^* = e_{hg^{\mone}} (e_1)^*\neq 0$ if and only if $g=h$. 
It is hence isomorphic to the injective $A$-module with socle $(e_g)^*$, which is $Ae_{\nu(g)}$.
\endproof


\begin{proposition}\label{prop61}
We have adjunctions
\begin{enumerate}[$($a$)$]
\item\label{prop61.1}
$(Ae_g\otimes_{\Bbbk}e_1A, Ae_{\nu(1)g^{\mone}}\otimes_{\Bbbk}e_1A)$;
\item\label{prop61.2}
$(A_0\bar e_g\otimes_{\Bbbk}e_1A, Ae_{\nu(1)g^{\mone}}\otimes_{\Bbbk}\bar e_1A_0)$;
\item\label{prop61.3}
$(Ae_g\otimes_{\Bbbk}\bar e_1A_0, A_0\bar e_{g^{\mone}}\otimes_{\Bbbk}e_1A)$;
\item\label{prop61.4}
$(A_0\bar e_g\otimes_{\Bbbk}\bar e_1A_0, A_0\bar e_{g^{\mone}}\otimes_{\Bbbk}\bar e_1A_0)$.
\end{enumerate}
\end{proposition}
\begin{proof}
By \cite[Theorem 4.3]{CM}, there is an equivalence of categories between the projective objects in $\tilde{\Y}$ and the $G$-invariant objects in the category of projective ${\hat A}$-${\hat A}$-bimodules. Under this equivalence $Ae_g\otimes_{\Bbbk}e_1A$ corresponds to $\bigoplus_{h \in G}Ae_{gh}\otimes_{\Bbbk}e_hA$. Hence its adjoint corresponds to the adjoint of $\bigoplus_{h \in G}Ae_{gh}\otimes_{\Bbbk}e_hA$, which is given by 
$\bigoplus_{h \in H}Ae_{\nu(h)}\otimes_{\Bbbk}e_{gh}A \cong \bigoplus_{h \in H}Ae_{\nu(g^{\mone}h)}\otimes_{\Bbbk}e_{h}A $. Under the equivalence, this corresponds to $Ae_{\nu(g^{\mone})}\otimes_{\Bbbk}e_{1}A$ which is isomorphic to $Ae_{\nu(1)g^{\mone}}\otimes_{\Bbbk}e_{1}A$ by Lemma \ref{nucompat}. The other adjunctions are checked similarly.
\end{proof}

\subsection{Simple transitive birepresentations of $\cG_A$}\label{sec:freetransitive}

We keep the notation and assumptions from Section \ref{s3.6}. In particular, we assume that $A$ is basic, self-injective has a fixed complete $G$-invariant
set $\mathtt{E}$ of primitive idempotents on which $G$ acts regularly.
It follows immediately from Proposition \ref{prop61} that, under these assumptions, $\cG_A$ and $\widetilde{\cG_A}$ are quasi fiab, and fiab if and only if $A$ is weakly symmetric.

In order to classify simple transitive birepresentations for $\cG_A$, we will need a slight generalisation of \cite[Theorem 15]{MMMZ}. 

\begin{theorem}\label{mmmzgen}
Let $\cC$ be a quasi multifiab bicategory, $\J$ a maximal two-sided cell in $\cC$ and $\H$ an $\H$-cell in $\J$ such that $\H^*=\H$. Let $\ti$ be such that, for all $\rF\in \H$, we have $\rF\in \cC({\ti,\ti})$ and let $\cC_\H$ be the $2$-full subbicategory of $\cC$ on object $\ti$ with $1$-morphisms in $\add\{\mathbbm{1}_\ti, \rF\,\vert \;\rF\in \H\}$. Then there is a bijection between simple transitive birepresentations of $ \cC$ with apex $\J$ and those of $\cC_\H$ with apex $\H$.
\end{theorem}

\proof
The proof is analogous to that of Theorem 4.32 in \cite{MMMTZ2}, noting that the only place where $\cC$ being fiab is crucially used there is to obtain $\H=\H^\star$.
\endproof

We thus obtain the analogous classification of simple transitive birepresentations to the one given in \cite{MMZ2} in the case where $G$ is abelian.
 
\begin{theorem}\label{thm11}
We retain the above assumptions on $A$. Let $\bfM$ be a simple transitive birepresentation of $\cG_A$.
\begin{enumerate}
\item If the apex of $\bfM$ is $\J_{1}$, then $\bfM$ is the trivial birepresentation associated to $\mathbbm{1}_\bullet$, meaning $\bfM(\bullet)$ is equivalent to $\Bbbk\lmod$, the identity $1$-morphism $(A_1, {\pi}_1)$ acts as the identity functor, and all other indecomposable $1$-morphisms annihilate.
\item If the apex of $\bfM$ is $\J_{0}$, there is a natural bijection between
equivalence classes of simple transitive birepresentations of $\cG_{A}$ with apex $\mathcal{J}_{0}$
and pairs $(K,\omega)$, where $K$ is a subgroup of $G$ and $\omega\in H^2(K,\Bbbk^*)$.
\end{enumerate}
\end{theorem}

\begin{proof}
For $\mathcal{J}=\mathcal{J}_{1}$, the statement is immediate, as the $\J_{1}$-simple quotient of $\cG_{A}$ is biequivalent to $\cC_\Bbbk$.

For $\mathcal{J}=\mathcal{J}_0$, consider $\widetilde{\cG_A}$. Then we can realize $\cG_{A}$ as a
both $1$- and $2$-full subbicategory of $\widetilde{\cG_A}$ given by the endomorphism category $\widetilde{\cG_A}_\bullet$ of $\bullet$. Moreover, $\J_0$ corresponds to $\H_{11}$ under this identification. By Theorem \ref{mmmzgen}
there is a bijection between equivalence classes of simple transitive birepresentations of
$\cG_A$ with apex $\J_0$, and equivalence classes of simple transitive birepresentations of $\widetilde{\cG_A}$
with apex $\widetilde{\J}_0$.

On the other hand, again, by Theorem \ref{mmmzgen}, there is a bijection between equivalence classes of simple transitive birepresentations of
$\widetilde{\cG_A}$
with apex $\widetilde{\J}_0$ and equivalence classes of simple transitive birepresentations of $\widetilde{\cG_A}_\ast$ with apex the unique two-sided cell $\H_{00}$. However, by Lemma \ref{vecGequiv}, the latter is biequivalent to $\cV ec_G$ and its simple transitive birepresentations are in bijection with pairs $(K,\omega)$, where $K$ is a subgroup of $G$ and $\omega\in H^2(K,\Bbbk^*)$ by \cite[Theorem~2]{Os}.
\end{proof}

\begin{remark}
Observe that Theorem \ref{thm11} implies that for a finite-dimensional radically graded basic  Hopf algebra $A$, the associated bicategory of symmetric bimodules $\cG_A$ only has finitely many simple transitive birepresentations up to equivalence. By contrast, the bicategory $\cH_A$, which by Theorem \ref{HAinGA} can be viewed as a $1$-full subbicategory of $\cG_A$, generally has infinitely many non-equivalent simple transitive birepresentations, see e.g. \cite[Theorem 4.10]{EO}. It would be interesting to investigate $1$-full subbicategories of $\cG_A$ which contain $\cH_A$ and try to determine where the jump from finitely to infinitely many equivalence classes of simple transitive birepresentations occurs.
\end{remark}

K.H.: {\tt katherina.hristova\symbol{64}gmail.com}

V.M.: {\tt v.miemietz\symbol{64}uea.ac.uk} \\School of Mathematics, University of East Anglia,
Norwich NR4 7TJ, UK.

\end{document}